\documentclass[onefignum,onetabnum, final]{siamart171218}
\usepackage{color,soul}
\pdfoutput=1

\usepackage{lipsum}
\usepackage{amsfonts}
\usepackage{graphicx}
\usepackage{mathtools}
\usepackage{algorithmic}

\ifpdf
  \DeclareGraphicsExtensions{.eps,.pdf,.png,.jpg}
\else
  \DeclareGraphicsExtensions{.eps}
\fi

\newsiamremark{remark}{Remark}
\newsiamremark{hypothesis}{Hypothesis}
\crefname{hypothesis}{Hypothesis}{Hypotheses}
\newsiamthm{claim}{Claim}

\headers{Parallel Skeletonization for Integral Eq. in Evolving Domains}{J. P. Ryan and A. Damle}

\title{Parallel Skeletonization for Integral Equations in Evolving Multiply-Connected Domains
}
\author{John Paul Ryan\thanks{Department of Computer Science, Cornell University, Ithaca, NY 14853 (~\email{\mbox{johnryan@cs.cornell.edu}},
  \email{\mbox{damle@cornell.edu}}).}
\and Anil Damle\footnotemark[1]}

\usepackage{amsopn}

\newcommand{\jrev}[1]{\textcolor{black}{#1}}
\newcommand{\secondrev}[1]{\textcolor{black}{#1}}
\newcommand{\thirdrev}[1]{\textcolor{black}{#1}}

\ifpdf
\hypersetup{
  pdftitle={Parallel Skeletonization for Integral Equations in Evolving Multiply-Connected Domains},
  pdfauthor={J. P. Ryan and A. Damle}
}
\fi

\begin{document}

\maketitle
\begin{abstract}
This paper presents a general method for applying hierarchical matrix skeletonization factorizations to the numerical solution of boundary integral equations with possibly rank-deficient integral operators. Rank-deficient operators arise in boundary integral approaches to elliptic partial differential equations with multiple boundary components, such as in the case of multiple vesicles in a viscous fluid flow. Our generalized skeletonization factorization retains the locality property afforded by the ``proxy point method,'' and allows for a parallelized implementation where different processors work on different parts of the boundary simultaneously. Further, when the boundary undergoes local geometric perturbations (such as movement of an interior hole), the \jrev{factorization} can be recomputed \jrev{efficiently with respect to the number of modified discretization nodes.} We present an application that leverages a parallel implementation of skeletonization with updates in a shape optimization regime. 
 \end{abstract}



\section{Introduction}
Consider the homogeneous boundary value problem
\[    Lu({x}) = 0 \qquad {x}\in\Omega
\]
\[
    u({x}) = f({x}) \qquad {x}\in\Gamma
\]
where $L$ is an elliptic differential operator, $f(x)$ is given Dirichlet data on the boundary $\Gamma$, and $u(x)$ is the desired solution in the interior domain $\Omega$. Assuming sufficient smoothness of $\Gamma$ and $f(x)$, the solution can often be efficiently numerically computed by considering associated boundary integral equations:
\begin{equation}\label{eq:basicbieforward}
u(x)=
\int_\Gamma K_{\Omega}({x},{y})\mu({y})\mathrm{d\Gamma}(y) \qquad {x} \in \Omega
\end{equation}
\begin{equation}\label{eq:basicbie}
\secondrev{\lambda \mu(x) +} \int_\Gamma K_\Gamma({x},{y})\mu({y})\mathrm{d\Gamma}(y)  = f({x})\qquad {x} \in \Gamma
\end{equation}
where $K_\Omega$ and $K_\Gamma$ are kernels related to the Green's function and $\mu(x)$ is an intermediary function to be calculated via~\eqref{eq:basicbie}. \secondrev{The scalar $\lambda$ depends on the choice of $K_{\Gamma}.$ When $\lambda=0$, \eqref{eq:basicbie} is a \emph{Fredholm equation of the first kind}, otherwise it is a \emph{Fredholm equation of the second kind}}. This method requires only a discretization of the boundary, and allows for the computation of the solution at any point inside the domain without the need to discretize the entire domain.

Using $K$, $\mu$, and $f$ to denote the discretizations of the integral operator with kernel function $K_\Gamma(x,y)$, the unknown function $\mu(x)$, and the given function $f(x)$ respectively, the discretization of~\eqref{eq:basicbie} is 
\begin{equation}
    \label{eq:biediscretee}
    \secondrev{(\lambda I+K) }\mu = f.
\end{equation}
Solving for $\mu$ given $f$ requires a linear solve involving the dense matrix $\secondrev{(\lambda I+K)}$, and for large discretizations of the boundary, general dense factorizations such as the \emph{LU} decomposition can be prohibitively expensive. For this reason, it is advantageous to consider properties of $K$ that enable more efficient factorizations.

In many cases, $K_\Gamma(x,y)$ satisfies an approximate separability condition:
\begin{equation}\label{eq:sep}
K_\Gamma({x},{y}) \approx \sum_{i=1}^p u_i({x})v_i({y})\qquad  |{x}-{y}|>\gamma.
\end{equation}
\secondrev{One setting in which this property arises is when} $K_\Gamma(x,y) = K_\Gamma(x-y)$ and this function smoothly decays away from the origin, as is the case for many non-oscillatory Green's functions.
For example, in Section~\ref{sec:stokeseq} we examine the following boundary integral equation for computing viscous fluid flow velocities:
\begin{equation}\label{eq:introeq}
 -\frac{1}{2}{\mu}({x}) +
    \frac{1}{\pi}\int_{\Gamma} \frac{({x}-{y})\cdot{n}({y})}{\|{x}-{y}\|_2^4} ({x}-{y})\otimes ({x}-{y}){\mu}({y})\mathrm{d\Gamma}({y}) =  {f}({x}) \qquad x\in\Gamma,
\end{equation}
where ${n}({y})$ is the normal vector to the boundary $\Gamma$ at ${y}$. Since the underlying kernel function satisfies~\eqref{eq:sep}, the discretized integral operator $K$ will have numerically low rank off-diagonal blocks. In Section~\ref{sec:fact} we outline a hierarchical matrix factorization that leverages this property to approximate $K$ by a product of easily invertible matrices \cite{rskel, strongrskel}. Given this approximate factorization of $K$, solving~\eqref{eq:biediscretee} can be done efficiently. Furthermore, as a direct method, this scheme is well suited to solve problems for many different right hand sides (for example, computing the solution in a fixed geometry for different boundary conditions).
 
When the boundary is multiply connected, \emph{i.e.,}\ $\Gamma=\Gamma_0\cup \Gamma_1\cup\dots\cup \Gamma_p$, the integral operator (and hence the matrix $K$) can become rank-deficient. In Section~\ref{sec:stokeseq} we show how this degeneracy arises, and in Section~\ref{sec:generalform} we develop a more general version of the skeletonization factorization applicable to these settings.

Throughout construction of the factorization, we compress off-diagonal blocks by decoupling sets of integration nodes based on analysis of kernel interactions. This compression is performed locally since it depends only on interactions with nearby integration nodes. The technique we use is known as the \emph{proxy point method} \cite{id, kernindfact, kernindfact23d, hifie, strongrskel, corona2013, mrdirect, rskel, direct3d, lineardirect, mindenupdate, stokessolver, direct3d, xing2019}, and is described in Section~\ref{sec:proxsec}. Local modifications to the underlying data $\Gamma$ can be made without requiring full recomputation of the factorization of $K$. Further, the structure of the factorization persists following such modifications, and has all the benefits of a factorization computed from scratch. Section~\ref{sec:factupdate} reviews such a method for updating hierarchical skeletonization factorizations \cite{mindenupdate}.

The fact that compression is performed locally and that updates to the factorization can be made quickly following geometric perturbations make this procedure particularly useful in optimization problems where solutions are calculated for many closely related boundaries. For example, this arises in objective function evaluation, search direction selection, line search, etc. In Section~\ref{sec:optex} we provide illustrative examples of this setting.

\subsection{Background}
The pioneering work in compressing discretizations of kernel matrices based on hierarchical rank structure is the development of the Fast Multipole Method (FMM) \cite{fmm}, which represents interactions between points at a distance by truncated multipole expansions. For problems where analytic expansions cannot be used, significant work has been done to develop kernel-independent methods \cite{kernindfact, kernindfact23d, rskel, hifie}. These methods perform compression algebraically by analyzing matrix entries directly (i.e.,\ requiring samples of kernel interactions instead of analytic expansions). The study of these techniques has involved a broad exploration of the purely algebraic properties of the matrices, instead of the underlying kernel function. In both the analytically and algebraically motivated regimes, these methods may be used to numerically solve linear systems, either via approximate direct solves or using them as preconditioners in iterative methods \cite{mhodlr, corona2013, mrdirect, direct3d, gil}.

From a purely algebraic perspective, matrices displaying off-diagonal rank structure have been analyzed extensively and may fall into many different classes. Hierarchical off-diagonal low rank (HODLR) matrices satisfy the property that, given a tree-structured partitioning of the indices of the matrix, submatrices corresponding to different partitions of indices at the same level of the tree have rank less than $k$, a constant for the whole matrix. Hierarchically semi-separable (HSS) matrices are HODLR matrices with the additional property that each parent node’s basis can be constructed from those of its children (often referred to as a \emph{nested bases}). $\mathcal{H}$ (and $\mathcal{H}^2$) matrices~\cite{hackbusch2015hierarchical,borm2003introduction,bebendorf2008hierarchical} are like HODLR (and HSS) matrices except that there is a so-called admissibility condition. Admissibility conditions dictate which pairs of partitions correspond to low-rank sub-matrices, such as those that are above a certain distance apart. Throughout this work we will assume our matrices satisfy the \emph{weak admissibility condition}, i.e.,\ all pairs of distinct partitions have low rank, and we refer the interested reader to \cite{strongrskel, ifmm1} for discussion on other admissibility conditions.

In settings where the solution is desired following small perturbations to the boundary geometry (for example, the rotation of a fin in a channel, or a modification to the shape of a wing on an airplane), several methods exist which efficiently solve related problems faster than completely refactoring the underlying matrix. One possibility is to treat the perturbations as low rank updates and apply the Sherman-Woodbury-Morrison formula~\cite{zhang2018fast,zhang2020alternative}, which takes advantage of the fact that the original factorization can be used for a fast linear solve. The advantage of the method in~\cite{zhang2018fast} \secondrev{(and recent improvements~\cite{zhang2020alternative})} is that no significant matrix factorizations need to be recomputed. On the other hand, this method will gradually slow as more low-rank updates are applied (say due to continuing \secondrev{permanent} changes to the boundary) and \secondrev{could} encounter conditioning issues over time\textemdash for this reason in some settings it is desirable to have an updating methodology that retains the form of the factorization while incorporating updates.

Minden, et al. \cite{mindenupdate} present a method for updating hierarchical skeletonizations which preserves the structure of the factorization at each updating step. This is the scheme that we appeal to here to deal with changes to domain boundaries. Importantly, here we focus on schemes applicable to integral equation formulations (other closely related, but distinct work, focuses \jrev{on time dependent problems}~\cite{wang2019fast,greengard2018accurate}). However, rank-structured techniques are also often applicable to discretizations of differential operators~\cite{ho2016hierarchicalDE} and in those settings alternative methodologies may be applicable~\cite{liu2018fast}.

\begin{remark}
\jrev{
For \secondrev{sufficiently} small updates or in settings where the number of \secondrev{permanent} updates \secondrev{to a fixed base geometry} is limited, the technique in~\cite{zhang2018fast} may exhibit better performance \secondrev{than the updating scheme from~\cite{mindenupdate}}, as certain work necessary to build a valid rank-structured factorization for the new problem is avoided. However, in~\cite{zhang2018fast} the authors also explicitly show regimes where using their updating technique may be slower than computing a new factorization. Similarly, the use of Sherman-Woodbury-Morrison means that solves with the updated factorization are slower than those with the base system. In contrast, by updating the factorization directly we avoid these two potential pitfalls. As we will discuss in Section~\ref{sec:comp_complex}, the worst case cost of our method is bounded from above by recomputing the factorization from scratch. This makes our method preferable when, e.g., a valid rank-structured factorization is desired for many different boundary configurations. Furthermore, once we have updated the factorization, linear solves are as fast as with a factorization built from scratch.}
\end{remark}

An important application of rank-structured solvers to problems with changing boundary geometries is in viscous fluid flow simulations \cite{stokessolver,vesicleflow, 3dvesicle}. For example, \cite{vesicleflow} uses a rank-structured solver for a microfluidic flow simulation, namely the deformation of inextensible vesicles in unbounded domains. Biros and Ying \cite{navier} have also demonstrated the applicability of these techniques in numerically solving unsteady Navier-Stokes equations. To our knowledge, no Stokes flow simulators exist that leverage a fast updating scheme to the rank-structured factorizations. Since the performance gains due to fast updating of the factorization accumulate over time, and applications can require $10^6$ or more \cite{vesicleflow} linear solves, we consider the benefits of a factorization updating scheme of substantial importance in this area.

\secondrev{The structure of these factorizations and the scale of problems addressed makes it natural to consider parallel versions of these algorithms. In fact, given the breadth of distinct, albeit closely related, factorization structures there is also extensive work on parallel implementations. The most closely related work to our own, based on factorization type, is~\cite{li2017distributed}. However, that work focuses on a distributed memory setting and sparsely discretized differential operators on regular grids. In contrast, our focus is on integral equations with general geometries. In fact, much of the closely related existing work on parallel implementations~\cite{wang2016parallel,chen2018distributed} focuses on sparse systems. Based on the types of problem being addressed, the most closely related work to our own is STRUMPACK~\cite{rouet2016distributed} for HSS matrices (allowing for the use of randomized schemes to construct the low-rank factors) and the recently developed H2Pack~\cite{huang2020h2pack} for $\mathcal{H}^2$ matrices (using the proxy point method). Our work focuses on distinct factorization formats, contains implementations of the factorization updating algorithms of~\cite{mindenupdate}, and is built to work with the augmented systems necessary for some of the problems we solve.}

\subsection{Contribution}
In this work we present a general formulation of a fast hierarchical matrix factorization applicable to boundary integral equations for simply or multiply connected domains. In addition, we demonstrate the ability of the factorization to be efficiently updated following changes to the boundary geometry. Previous work in this setting experimented with small boundary or coefficient perturbations (e.g., adding a small bump to a simply connected boundary \jrev{\cite{mindenupdate}}) that minimally affect the solution. However, \secondrev{the techniques initially introduced in~\cite{mindenupdate} can be leveraged more generally to allow for the deformation, movement, addition, and deletion of interior holes. In this work, we clearly show the breadth of problem settings to which these techniques can be applied and, most notably, we clearly show how the scheme is viable in settings where the solution significantly changes throughout the domain based on local boundary updates.} This is most evident we solve problems such as Stokes flow or steady-state heat with Neumann conditions in multiply-connected domains with moving interior holes. This represents a powerful new technique for the numerical solution of elliptic PDEs with many related boundary configurations and conditions.

We also present an implementation of the hierarchical factorization which parallelizes compression of nodes at each level of the tree decomposition. Arguing that this level of parallelism is most powerful in the initial factorization, we demonstrate how computational resources can be more efficiently allocated in problems where \jrev{many consecutive local updates} to the factorization are required. This is the case in optimization settings where computing the ideal shape and/or size of interior holes is desired. Our work includes a novel demonstration of the benefits of hierarchical matrix factorization updating for such optimization problems, and this technique holds great promise for related time-dependent simulations, such as the simulation of vesicles in capillary flows. Our accompanying implementation allows for \jrev{easy design} of multiply-connected boundaries \jrev{via user-specified spline knots} and optimization given a user-specified objective function. 

\jrev{The core contribution of this manuscript is, ultimately, the novel combination of all the above pieces---development of the factorization and updating schemes adapted to the formulation of the integral equations applicable to multiply connected domains, implementation of a parallel code amenable to two and three dimensional problems, and numerical experiments illustrating efficacy of the implemented methods and efficiency of the developed software.}

\section{Preliminaries}\label{sec:prelim}
\subsection{Notation}

Throughout this manuscript we will use the matrix subscript notation $K_{AB}$ to refer to the submatrix of $K$ which contains the rows indexed by the index set $A$ and the columns indexed by the index set $B$. We will refer to a submatrix which contains the rows indexed by $A$ and all of the columns as $K_{A, :}$ and, analogously, $K_{:,B}$ is the submatrix which contains the columns indexed by $B$ and all of the rows.
\subsection{Hierarchical matrix factorization}\label{sec:fact}
The factorization we use is based on \cite{rskel}, and is applicable to matrices that arise from kernel discretizations where kernel interactions between distinct boxes of
points are numerically low-rank (i.e., the kernel function satisfies~\eqref{eq:sep}). We begin by showing how to approximate such a kernel matrix as the product of easily invertible sparse matrices.

As a starting point, consider the set of points $B\cup F$ comprised of two disjoint subsets of points indexed by $B$ and $F$ with $|B|\ll|F|$ (think of $B$ as corresponding to points inside a box and $F$ as corresponding to many points outside the box). The matrix of kernel interactions may be written as 
\begin{equation}\label{eq:weak}
K = \begin{bmatrix}
K_{BB} & K_{BF} \\ 
K_{FB} & K_{FF} 
\end{bmatrix}.
\end{equation}
\jrev{We will assume throughout this work that our kernel satisfies the weak admissibility condition\footnote{
If the kernel only satisfies the \emph{strong admissibility condition}, then \eqref{eq:weak} would need blocks for non-low-rank interactions between points inside the box and points outside but nearby, see \cite{strongrskel}.
}, so $K_{FB}$ is numerically low-rank} and we may construct a low-rank approximation to $K_{FB}$ 
\begin{equation}\label{eq:lowrank}
K_{FB} \approx WZ^T.\end{equation}
A popular choice in this setting is the interpolative decomposition \cite{id}
\[W=K_{FS}\qquad Z^T = [I \quad T],\]
where $B$ is partitioned as
\[B =\{S\quad R\}\]
and we omit permutations of indices for the sake of exposition. In this partition,  $S$ and $R$ stand for ``skeleton'' and ``redundant'' indices, respectively. The key idea is that we select the most important subset of indices $S$ in $B$. We then represent interactions with points indexed by $R$ as linear combinations of the important columns. Practically, $S$ and $T$ can be found via a column-pivoted QR factorization \cite{golubqr}, and can be chosen so that~\eqref{eq:lowrank} holds to any desired accuracy $\varepsilon$. We can use this factorization to decouple the redundant points by eliminating the ${FR}$ and ${RF}$ subblocks of $K$\textemdash this can be done by applying interpolation matrices on the left and right in the following manner
\begin{equation}\label{eq:tmat}
\begin{bmatrix}
X_{RR} & X_{RS} & 0 \\ 
X_{SR} & K_{SS} & K_{SF} \\ 
0 & K_{FS} & K_{FF} \\ 
\end{bmatrix} \approx
\begin{bmatrix}
I & -T^T & 0 \\ 
0 & I & 0 \\ 
0 & 0 & I \\ 
\end{bmatrix}\begin{bmatrix}
K_{RR} & K_{RS} & K_{RF} \\ 
K_{SR} & K_{SS} & K_{SF} \\ 
K_{FR} & K_{FS} & K_{FF} \\ 
\end{bmatrix}   \begin{bmatrix}
I & 0 & 0 \\ 
-T & I & 0 \\ 
0 & 0 & I \\ 
\end{bmatrix}.
\end{equation}
We use $X$ to denote a matrix block that has been modified, and we have assumed $K$ is symmetric (Section~\ref{sec:nonsym} discusses nonsymmetric kernels). We may now use block Gaussian elimination to decouple $R$ from the rest of the matrix via
\begin{equation}
\begin{bmatrix}
X_{RR} & 0 & 0 \\ 
0 & X_{SS} & K_{SF} \\ 
0 & K_{FS} & K_{FF} \\ 
\end{bmatrix}
=
G_L
\begin{bmatrix}
X_{RR} & X_{RS} & 0 \\ 
X_{SR} & K_{SS} & K_{SF} \\ 
0 & K_{FS} & K_{FF} \\ 
\end{bmatrix} 
G_R,
\end{equation}
where
\begin{equation}\label{eq:schurmat}
G_L = 
\begin{bmatrix}
I & 0 & 0 \\ 
-X_{SR}X_{RR}^{-1} & I & 0 \\ 
0 & 0 & I \\ 
\end{bmatrix}\qquad
G_R = 
\begin{bmatrix}
I & -X_{RR}^{-1}X_{RS} & 0 \\ 
0 & I & 0 \\ 
0 & 0 & I \\ 
\end{bmatrix}
.
\end{equation}
\jrev{Note that this step involves the inversion of $X_{RR}$---if necessary, we can guarantee $X_{RR}$ is sufficiently well-conditioned by using pivoted factorizations, and moving indices from $R$ into $S$ as needed.}

In summary, by applying easily invertible sparse matrices on the left and right of $K$ we construct a sparser matrix. Importantly, the $K_{FF}$ block remains unchanged and it may be further compressed in the same way. Ultimately, our goal is to iteratively apply this method to sparsify $K$ as much as possible. To that end, we must discuss how to iteratively choose the index sets $B$ and $F$. 
\subsection{Interaction matrices and domain decompositions}
Common sources of hierarchical matrices are problems where entries $K_{ij}$ are kernel interactions between points $x_i$ and $x_j$ in space. Some examples include covariance matrices in Gaussian processes and kernel matrices in boundary integral equation discretizations \cite{mle, martinssonbook}. In these settings, we form a tree decomposition of the domain that reveals compressible subblocks based on relationships between tree nodes. 

Suppose we wish to factor the interaction matrix of a set of points $x_1,...,x_N$ along a circle (see Figure~\ref{fig:treefig}). We first form an adaptive tree decomposition of the domain, such that the leaves of our tree contain no more than a prescribed constant number of points. Working first at the leaf level of the tree, we let $B_i$ be the set of all indices corresponding to points in the $i$th leaf box.  We then let $F_i$ be the set of all indices of points outside of the $i$th leaf box so that we may compress

\begin{figure}[t]
  \centering
  \includegraphics[scale=0.27]{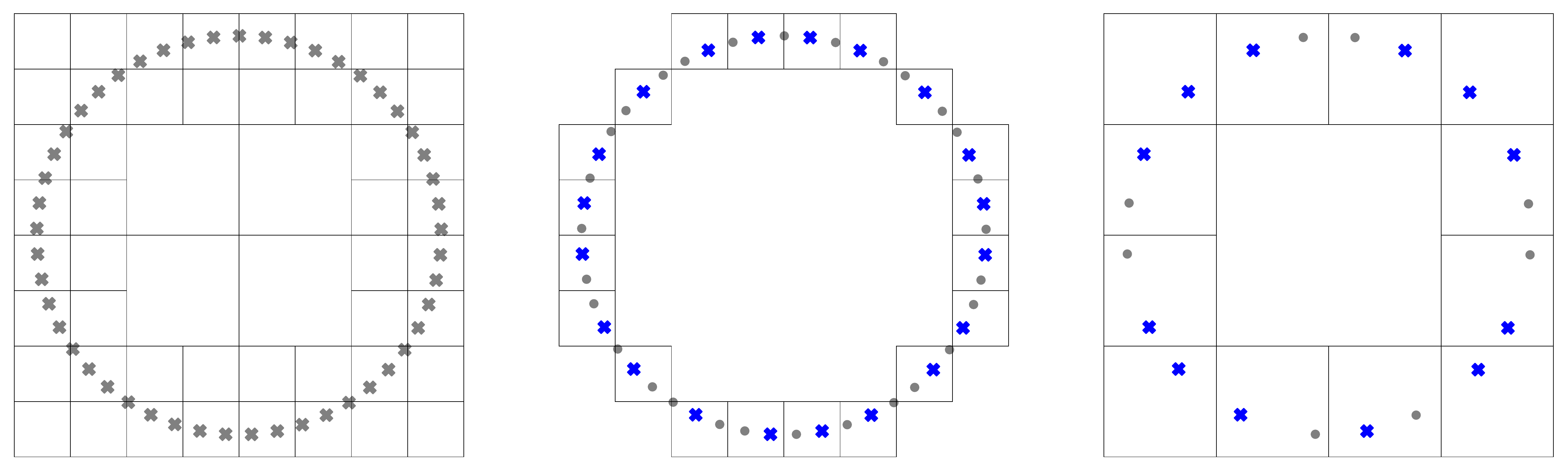}
  \caption{Left: we start by building an adaptive tree decomposition of the domain where the leaf level nodes contain fewer than some prescribed amount of points (grey x's). Center: focusing first on the nonempty leaf nodes with points inside, we perform the compression~\eqref{eq:firstcompress} for every box at the leaf level, partitioning the points in each box into skeleton points (blue x's) and redundant points (grey circles). Right: we continue compression by moving up one level in the tree and considering only the skeleton points from the previous level. We then perform the same compression procedure as before for the new boxes, thereby further shrinking the total number of skeleton points.f 
  Note that the above figures are purely illustrative\textemdash in reality, the ratio of redundant points to skeleton points is typically far greater.}\label{fig:treefig}
\end{figure}

\begin{figure}[t]
  \centering
 \includegraphics[scale=0.4]{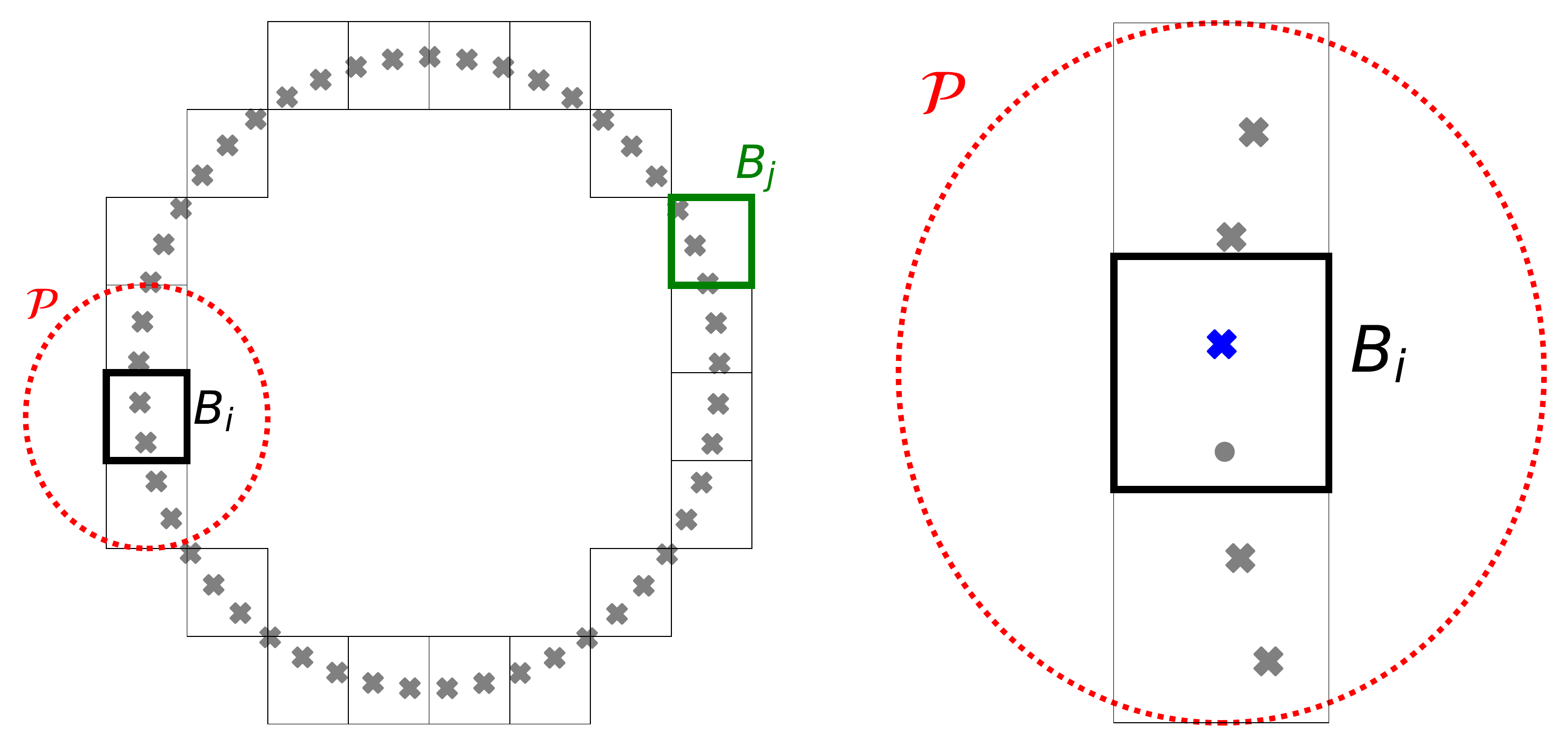}
  \caption{ 
  Left: for many kernels, the interactions between points in a box $B$ (bold) and distant points can be well approximated by considering interactions between the interior points and points on a ``proxy'' circle $\mathcal{P}$ (red). Section~\ref{sec:proxsec} describes how to compute compression matrices for this box using only points on and inside of $\mathcal{P}$. Right: the only information needed to perform compression for a box $B$ are its points' interactions with nearby points and with points along a proxy circle. Importantly, points in distant boxes (for example, in the green box $B_j$ in the left figure) can be added, deleted, or moved around, and the compression matrices and skeleton/redundant partitions for $B$ are still valid. 
  }\label{fig:pxyfig}
\end{figure}

\begin{equation}\label{eq:setup}
K = 
\begin{bmatrix}
K_{B_iB_i} & K_{B_iF_i} \\ 
K_{F_iB_i} & K_{F_iF_i}  \\ 
\end{bmatrix}
\end{equation}
in the following way
\begin{equation}\label{eq:firstcompress}
\begin{bmatrix}
X_{R_iR_i} & 0 & 0 \\ 
0 & X_{S_iS_i} & K_{S_iF_i} \\ 
0 & K_{F_iS_i} & K_{F_iF_i} \\ 
\end{bmatrix}
 \approx
U_i
K  
V_i,
\end{equation}
where the $U_i$ and $V_i$ matrices are products of the block unitriangular compression matrices in~\eqref{eq:tmat} and~\eqref{eq:schurmat}. We will use $B_i$ to refer both to the index set and its corresponding box in the tree, and we will refer to the above procedure as ``compressing box $B_i$.'' In practice we do not explicitly form $U_i$ and $V_i$ but instead apply the constituent matrices in sequence taking advantage of their block unitriangular structure. 

We may then consider $K_{F_iF_i}$ and compress every other box on this level, yielding

\begin{equation}\label{eq:levelcompress}
\begin{bmatrix}
X_{R_1R_1} & 0 & 0 & 0 & 0 & 0 & 0 \\ 
0 & \ddots & 0 & 0 & 0 & 0 & 0\\
0 & 0 & X_{R_mR_m} & 0  & 0 & 0 & 0\\
0 & 0 & 0 & X_{S_1S_1} & K_{S_1S_2} & \dots &  K_{S_1S_m} \\
0 & 0 & 0 &  K_{S_2S_1} & X_{S_2S_2} & \dots &  K_{S_2S_m} \\ 
0 & 0 & 0 &  \vdots & \vdots & \ddots & \vdots\\ 
0 & 0 & 0 &  K_{S_mS_1} & K_{S_mS_2} & \dots & X_{S_mS_m}\\ 
\end{bmatrix}
 \approx
U^{(l)}
K  
V^{(l)},
\end{equation} 
where $l$ is the level number (in \eqref{eq:levelcompress} we have $l=L$ where $L$ is the depth of the tree), $m$ is the number of boxes at level $l$ and 
\begin{equation}\label{eq:prods}
    U^{(l)}=\prod_{i=1}^m U_{m-i+1}\qquad V^{(l)}=\prod_{i=1}^m V_i.
\end{equation}

The bottom-right block of~\eqref{eq:levelcompress} is an interaction matrix between the skeleton points that haven't yet been compressed after level $l$, with some modifications to the diagonal blocks. Since the process of~\eqref{eq:setup}-\eqref{eq:firstcompress} relies only on the compressibility of the off-diagonal blocks, we may recurse on this set of skeleton points by moving up a level in the tree (see Figure~\ref{fig:treefig}, right).

Recursing until there are no more far-field interactions to compress (i.e., reaching the root of the tree) yields the factorization
\begin{equation}\label{eq:ddef}
X_D\coloneqq
\begin{bmatrix}
X_{R_1R_1} & 0 & 0 & 0 \\ 
0 & \ddots & 0 & 0 \\
0 & 0 & X_{R_nR_n} & 0  \\
0 & 0 & 0 & X_{SS} \\
\end{bmatrix}
 \approx
U
K  
V,
\end{equation}    
where
\begin{equation}\label{eq:bigprods}
    U=\prod_{l=1}^L U^{(L-l+1)}\qquad V=\prod_{l=1}^L V^{(l)}.
\end{equation}
\jrev{Note that $l=0$ is not included as no compression is performed at the root of the tree.}

In summary, we have produced an approximate factorization 
\[
K\approx \hat{K} = U^{-1}X_DV^{-1},
\]
and an approximate solution to $Kx=b$ may be computed via
\begin{equation}\label{eq:kinv}
    x = K^{-1}b \approx \hat{K}^{-1}b = VX_D^{-1}Ub,
\end{equation}
where factorizations of blocks on the diagonal $X_D$ are computed at factor time, not solve time. The overall accuracy of these approximations is closely related to the prescribed accuracy of the interpolative decompositions, see \cite{lineardirect, rskel, hifie, mrdirect} for further discussion. Notably, the ability to control the accuracy of this factorization also allows for the construction of a good \jrev{preconditioner for an iterative method (see, e.g.,~\cite{hifie} for numerical examples using the factorization as a preconditioner, and~\cite{saad} for further discussion on preconditioned iterative methods). }


\subsection{Nonsymmetric matrices}\label{sec:nonsym}
In Section~\ref{sec:fact}, we only use symmetry of $K$ to compress $K_{FB}$ and $K_{BF}$ with the same $T$ matrix. If $K$ is not symmetric, we can still find index sets $S$ and $R$ and a compression matrix $T$ by computing the interpolative decomposition
\[
\begin{bmatrix}
    K_{FS} \\
    K_{SF}^T
\end{bmatrix}T  \approx
\begin{bmatrix}
    K_{FR}\\
    K_{RF}^T
\end{bmatrix}.
\]
This procedure is simple to implement, works well in practice, and avoids having to keep track of distinct skeleton sets and interpolation matrices for each direction.
\subsection{Proxy \jrev{surface} method}\label{sec:proxsec}
A prohibitively costly component of the compression scheme in Section~\ref{sec:fact} is the formation and column-pivoted QR factorization of $K_{FB}$. If $K$ comes from a boundary integral equation corresponding to an elliptic PDE with constant coefficients, we can apply the following widely used remedy \cite{id, kernindfact, kernindfact23d, hifie, strongrskel, corona2013, mrdirect, rskel, direct3d, lineardirect, mindenupdate, stokessolver, xing2019}. Consider drawing a circle $\mathcal{P}$ around box $B$ (see Figure~\ref{fig:pxyfig}, left) and partitioning the index set $F$ into $F=F_{int}\cup F_{ext}$, where $F_{int}$ are indices corresponding to points inside $\mathcal{P}$ and $F_{ext}$ are outside. We can then write
\begin{equation}\label{eq:greens}
K_{FB} = 
\begin{bmatrix}
K_{F_{int}B}\\ 
K_{F_{ext}B}
\end{bmatrix}
 \approx
 \begin{bmatrix}
I & 0\\ 
0 & M_{F_{ext}P}
\end{bmatrix}
\begin{bmatrix}
K_{F_{int}B}\\ 
K_{PB}
\end{bmatrix}
\end{equation}   
\jrev{where $P$ refers to a set of discretization nodes on $\mathcal{P}$ and} the approximate representation $K_{F_{ext} B} \approx M_{F_{ext}P}K_{P B}$ is derived from discretizing 
\begin{equation}
    K(x-z) = \int_\mathcal{P} K(x-y) \jrev{\phi_z(y)}\mathrm{d}y.
\end{equation}
\jrev{where $z$ is any point in $F_{ext}$.}
The existence of this representation follows from Green's theorem (\jrev{for, e.g., oscillatory kernels, some care must be taken via concentric proxy surfaces or single and double layer representations, see \cite{mrdirect, rskel, hifie} for further discussion)}. Instead of performing a costly compression of $K_{FB}$, we compress the asymptotically smaller matrix on the RHS of~\eqref{eq:greens}

\begin{equation}\label{eq:gammaT}
\begin{bmatrix}
K_{F_{int}B}\\ 
K_{P B}
\end{bmatrix} \approx
\begin{bmatrix}
K_{F_{int}S}\\ 
K_{P S}
\end{bmatrix}
\begin{bmatrix}
I & T
\end{bmatrix}.
\end{equation}
Using the interpolative decomposition \eqref{eq:gammaT} in conjunction with~\eqref{eq:greens}, we see that the index set $S$ and interpolation matrix $T$ also compress the matrix $K_{FB}$. The additional error we incur from the integral discretization is related to the size of the discretization and the behavior of the kernel. In most cases, we fix the number of discretization nodes on $\mathcal{P}$ so that $|P|\ll |F_{ext}|$ and the integral discretization error is negligible compared to the interpolative decomposition error.

Although the above is sufficient for this work, we note that the use of proxy points can be extended to more general kernels. For example, common kernels arising when working with kernelized Gaussian processes necessitate the use of a proxy annulus \cite{mle}. Xing and Chow \cite{xing2019} present an algorithmic procedure for selecting proxy points for a given kernel. 

\subsection{Computational complexity}
\label{sec:comp_complex}
\jrev{The computational complexity of the recursive skeletonization factorization has been analyzed in \cite{rskel,hifie} in the context of integral equations for elliptic operators\textemdash here we review some of their results as we will use them later. Let $k_l$ refer to the maximal number of skeleton indices in a box on level $l.$ By assumption, for the weakly admissible problems of interest here $k_l$ can be bounded as 
\begin{equation}\label{eq:kl}
    k_l = 
    \begin{cases}
        \mathcal{O}(L-l) & d=1 \\
        \mathcal{O}(2^{(d-1)(L-l)}) & d>1
    \end{cases},
\end{equation}
where $d$ is the intrinsic dimension of the boundary and $L$ is the depth of the tree.
This follows from the observation that $k_l$ is on the order of the interaction rank between two adjacent blocks at level $l$. Assuming this growth leads to the following result on the complexity of constructing the factorization and subsequent matrix-vector multiplications or linear solves.}

\begin{theorem}[from \cite{rskel,hifie}]\label{thm:rscost}
\jrev{Assuming~\eqref{eq:kl} holds, the computational cost $t_f$ of the recursive skeletonization factorization for an $N$ node discretization of a boundary integral equation with a boundary of intrinsic dimension $d$ is 
\begin{equation}\label{eq:rscost}
t_f = 
        \begin{cases}
        \mathcal{O}(N) & d=1 \\
        \mathcal{O}(N^{3(1-1/d)}) & d>1
    \end{cases}.
\end{equation}
Furthermore, the cost $t_f^{(l)}$ associated with level $l$ in the factorization is 
\begin{equation}\label{eq:rskellevel}
t_f^{(l)} = 
        \begin{cases}
        2^{dl}\mathcal{O}(k_l^3) & l<L \\
        2^{dL}\mathcal{O}(c^3) & l=L
    \end{cases},
\end{equation}
where $c$ is the maximum number of indices in a box at the leaf level.
The cost of then applying the factorization to a vector or performing a linear solve is
\begin{equation}\label{eq:rsascost}
t_{a/s} = 
        \begin{cases}
        \mathcal{O}(N) & d=1 \\
        \mathcal{O}(N\log{N}) & d=2\\
        \mathcal{O}(N^{2(1-1/d)}) & d>2
    \end{cases}.
\end{equation}}
\end{theorem}

\jrev{In practice, skeleton points tend to cluster around the interfaces of boxes. Consequently, in~\cite{hifie}, Ho and Ying introduce the hierarchical interpolative factorization for integral equations (HIF-IE), which extends the above recursive skeletonization. Specifically, HIF-IE includes additional levels in the tree to further compress the interfaces between boxes before stepping up to a coarser level of the tree, resulting in slower growth in $k_l$. Their work results in better behavior of factorization and application costs, and the gains are strongly supported by experimental evidence. \secondrev{Although our experiments do not leverage this extension, we briefly review the improved complexities achieved by HIF-IE}. 
\begin{theorem}[from \cite{hifie}]\label{thm:hifie} Assuming $k_l=\mathcal{O}(L-l)$, then the cost $t_f$ of computing an HIF-IE factorization for an $N$ node discretization of a boundary integral equation with a boundary of intrinsic dimension $d$ is given by
\begin{equation}\label{eq:hifcost}
t_f =  \begin{cases}
        \mathcal{O}(N\log{N}) & d=2 \\
        \mathcal{O}(N\log^6{N}) & d=3
    \end{cases}.
\end{equation}
In particular, the cost $t_f^{(l)}$ associated with level $l$ in the factorization is 
\begin{equation}\label{eq:hiflevel}
t_f^{(l)} = 
        \begin{cases}
        2^{dl}\mathcal{O}(k_l^3) & l<L \\
        \mathcal{O}(2^{dL}c^3) & l=L
    \end{cases}.
\end{equation}
The cost of then applying the factorization to a vector or performing a linear solve is given by
\begin{equation}\label{eq:hifascost}
t_{a/s} = 
        \begin{cases}
        \mathcal{O}(N\log{\log{N}}) & d=2 \\
        \mathcal{O}(N\log^2{N}) & d=3
    \end{cases}.
\end{equation}
\end{theorem}}

\subsection{Parallelization}\label{sec:parallel}
Consider the impact of the order in which we compress boxes on a level. After compressing the first box on the leaf level, the factorization looks like 
\begin{equation}\label{eq:parfact}
\begin{bmatrix}
X_{RR} & 0 & 0 \\ 
0 & X_{SS} & K_{SF} \\ 
0 & K_{FS} & K_{FF} \\ 
\end{bmatrix}
\approx
U
\begin{bmatrix}
K_{RR} & K_{RS} & K_{RF} \\ 
K_{SR} & K_{SS} & K_{SF} \\ 
K_{FR} & K_{FS} & K_{FF} \\ 
\end{bmatrix} 
V.
\end{equation}
In particular the data required to compress subsequent boxes at the leaf level is nearly the same. The only exception is the introduction of zeros in the $K_{RF}$ and $K_{FR}$ blocks. 

This observation allows us to compress every box on a level in parallel. To see this, notice that given $T$ from an interpolative decomposition of the dense matrix
\[
\begin{bmatrix}
A \\ B 
\end{bmatrix} 
\approx
\begin{bmatrix}
A_{skel} \\ B_{skel}
\end{bmatrix}
\begin{bmatrix}
I &  T
\end{bmatrix},
\]
the same $T$ also satisfies
\[
\begin{bmatrix}
A \\ 0 
\end{bmatrix} 
\approx
\begin{bmatrix}
A_{skel} \\ 0
\end{bmatrix}
\begin{bmatrix}
I & T
\end{bmatrix}.
\]
Consequently, we may compress all boxes on a given level in parallel (neglecting to propagate the zeros in the LHS of~\eqref{eq:parfact} until moving onto the next level) and still achieve an accurate factorization. 

A trade-off is that there may be less compression when parallelizing this way since the skeleton sets are found by solving slightly larger interpolation problems than necessary. Fortunately, use of the proxy surface mitigates this concern\textemdash zeros introduced outside the proxy surface of a given box have no effect on its compression (see Figure~\ref{fig:pxyfig}, right). \jrev{The interpolation problem is only made larger by the consideration of interactions between the box's indices and indices in $F_{int}$ which may have been zeroed out by the compression of other boxes on the same level. In practice the effect on compression is small, and the speedup due to parallelizing an entire level outweighs the cost due to marginally less efficient compression of boxes.} 

We make note of one further opportunity for parallelism that exists regardless of whether or not the above parallel compression scheme is employed during factorization. Consider applying $V^{(l)}$ in~\eqref{eq:prods} to a vector
\[
V^{(l)}x=\left(\prod_{i=1}^m V_{i}\right)x.
\]
The block unitriangular matrices $V_i$ act on disjoint sets of components of $x$ (in particular, $V_i$ acts only on those $x_j$ where $j\in B_i$), and hence can be applied in parallel. The application of $U^{(l)}$ can also be parallelized for the same reason, as can $X_D$ as a block diagonal matrix, and also the inverses of these matrices. 
\subsection{Fast updating after perturbation}\label{sec:factupdate}
One major benefit of this form of factorization is that if discretization nodes are added, deleted, or moved around in a small area of the domain (such that relatively few of the leaf-node boxes are affected) the factorization can be updated with relative ease while maintaining the same structure. 

Suppose that some points in the green box of Figure~\ref{fig:pxyfig}, left, are perturbed. Due to the locality afforded by the proxy method, compression of the bold-faced box $B_i$ is completely unaffected. In other words, compression matrices associated with $B_i$ ($X_{R_i,R_i}$, $X_{S_i,S_i}$, $U_i$ and $V_i$) will be the same before and after modifications to points inside the green box $B_j$, and hence we needn't recompute them. Similarly, corresponding blocks along the diagonal of $X_D^{-1}$ in \eqref{eq:kinv} are unchanged, and their $LU$ factorizations may be reused. \jrev{As we move up the tree we only need to perform updates for boxes that contain modified points or those sufficiently close to them. We omit the detailed rules for selecting boxes for recomputation, but see \cite{mindenupdate} for further discussion. Importantly, the only boxes in need of recompression are $B_j$'s ancestors and a small number of boxes nearby them and the cost of these recomputations is summarized in Lemma~\ref{lem:updcost} for modifications that only affect a single leaf node.}
\begin{lemma}\label{lem:updcost}
\jrev{
Suppose points are added to, removed from, and/or modified within a single box $B_j$ at the leaf level after an initial recursive skeletonization factorization is performed. The cost $t_{uf}^{(l)}$ of recomputing the block matrices level $l$ for the updated factorization is 
\begin{equation}\label{eq:rskelupdlevel}
t_{uf}^{(l)} = 
        \begin{cases}
        \mathcal{O}(k_0^3) & l=0 \\
        M_l^{d}\mathcal{O}(k_l^3) & 0<l<L \\
        M_L^d\mathcal{O}(c^3) & l=L
    \end{cases},
\end{equation}
where $M_l$ is a small constant independent of $N$ and $L$ that is bounded from above by $5$ for $d\leq 3$ and $c$ is the number of points in the modified leaf node. \secondrev{When} $k_l=\mathcal{O}(L-l)$, the total cost of updating the factorization is $t_{uf}=\mathcal{O}(m\log^4{N})$, where $m$ is the number of perturbed points.}
\end{lemma}
\begin{proof}
\jrev{
The root box will always be an ancestor of $B_j$ and need to be refactored, giving the cost at $l=0.$ Similarly, at the leaf level $l=L$ only the box with updated points and its neighbors need to be refactored, each at a cost of $\mathcal{O}(c^3).$ At other levels, from~\cite{mindenupdate} the number of boxes needing recomputation is bounded from above by $5^d$ for $d\leq 3,$ each at a cost of $\mathcal{O}(k_l^3).$ Finally, the proof of the total cost is given in \cite{mindenupdate}.}
\end{proof}
\jrev{If $d=1$ and $k_l=\mathcal{O}(L-l)$ we observe that the cost of updating a single leaf level box is polylogarithmic in $N$. However, in the worst case updates across sufficiently many leaf level boxes will result in complexity equivalent to recomputing the factorization from scratch. For $d>1$ \secondrev{the cost is dominated by factoring at the root node, and is asymptotically the same as a refactorization from scratch. Nevertheless, Section~\ref{sec:results} demonstrates that there is still notable gains seen in practice when the number of affected leaf level boxes is relatively small. If the recompression techniques of HIF-IE are used to control the growth of $k_l$ to be $\mathcal{O}(L-l)$, then the asymptotic cost of updating is given by Lemma~\ref{lem:updcost}, although the practical gains seen are dependent on problem geometry (see \cite{mindenupdate} for a 2D volumetric experiment).} In either case, this updating scheme requires recomputation of the $X_{SS}$ block's factorization; a more desireable technique would not require recomputations near the top of the tree and is the subject of ongoing research.}

\begin{remark}
\jrev{
Implicit in the way our factorization is updated, the cost of updating a factorization can never exceed the cost of computing one from scratch. In practice we simply mark boxes that require updating and then refactor only those nodes in the tree. In the worst case, all the nodes of the tree are marked and we end up recomputing the factorization entirely. Importantly, that means using the procedure from~\cite{mindenupdate} is strictly beneficial\textemdash we save time when possible and fall back on recomputing a full factorization when necessary. In contrast, methods based on augmented systems such as~\cite{zhang2018fast} require determining when it is beneficial to recompute the base factorization.}
\end{remark}

\jrev{As described in \cite{mindenupdate}, the factorization that results from updating (with careful bookkeeping and tree refinement, and assuming the same size/location of the root node) will be \emph{exactly the same} as one would have from factoring anew. Without dynamic tree maintenance or in the case that the root node is shifted, the resulting factorization will be different, although just as accurate, as a new factorization from scratch. }

An important impact of the locality of this updating scheme is that there is less opportunity for parallelism due to the relatively small number of boxes needing recompression per level \jrev{(compare $M_l^d$ in \eqref{eq:rskelupdlevel} with $2^{dl}$ in \eqref{eq:rskellevel})}. We experimentally explore the consequences of these details in Section~\ref{sec:optex}.

\section{Problem description}\label{sec:probdesc}
The procedures described in Section~\ref{sec:prelim} are applicable to a broad range of problems in which certain kernel functions arise. In this work, we mainly focus on their applicability to the computation of steady viscous fluid flow given velocity boundary conditions in a 2D domain. This requires developing a boundary integral formulation of the relevant Stokes equations. We then use these techniques to numerically solve a discretized version of the problem. In this section we describe the relevant differential equations and their representation as integral equations on the boundary. 
\subsection{Stokes equations}\label{sec:stokeseq}
The Stokes equations
\begin{equation}\label{eq:stokes}
    -\Delta {u}({x}) + \nabla p({x}) = 0, \qquad
    \nabla \cdot {u}({x}) = 0,
\end{equation}
describe steady-state fluid velocity ${u}(x)\in\mathbb{R}^2$ and pressure $p(x)\in\mathbb{R}$ in incompressible flows where inertial forces are very small compared to viscous forces. They come from taking the limit of the non-dimensionalized Navier-Stokes equations as the Reynold's number vanishes. 

Consider the interior boundary value problem with boundary $\Gamma$ and corresponding Dirichlet data
\begin{equation}\label{eq:boundarycondition}
    {u}({x}) = {f}({x}), \qquad {x}\in \Gamma.
\end{equation}
Due to the linearity of \eqref{eq:stokes}, we may (under mild assumptions, see \cite{HsiaoBIE} Section 2.3) solve for the fluid velocity $u(x)$ on the interior $\Omega$ of $\Gamma$ by solving a set of corresponding boundary integral equations. 

The boundary integral approach represents the solution inside the domain as a boundary integral of some \emph{single} or \emph{double layer potential}. In electrostatics this is analogous to representing an electric field as an integral of a charge (single layer) or dipole (double layer) density on a surface. Following \cite{stokessolver, power, degenerate} we use the \emph{stresslet}\footnote{The stresslet is the symmetric part of the first moment of force for the Stokes equations.} as our kernel function, and the resulting boundary integral is
\begin{equation} \label{eq:stokesbie}
    {u}({x}) = \frac{1}{\pi}\int_{\Gamma} \frac{({x}-{y})\cdot{n}({y})}{\|{x}-{y}\|_2^4} ({x}-{y})\otimes ({x}-{y}){\mu}({y})\mathrm{d\Gamma}({y}) \qquad x\in\Omega.
\end{equation}
Plugging~\eqref{eq:stokesbie} into~\eqref{eq:boundarycondition} and taking the limit as $x$ approaches a point on $\Gamma$ results in 
\begin{equation}\label{eq:jump}
    {f}({x}) = -\frac{1}{2}{\mu}({x}) +
    \frac{1}{\pi}\int_{\Gamma} \frac{({x}-{y})\cdot{n}({y})}{\|{x}-{y}\|_2^4} ({x}-{y})\otimes ({x}-{y}){\mu}({y})\mathrm{d\Gamma}({y}) \qquad x\in\Gamma.
\end{equation}
The $-\frac{1}{2}\mu(x)$ is the result of a jump relation of the double layer potential as it approaches the boundary.

Solutions to~\eqref{eq:jump} exist but are not unique\textemdash the linear operator on the RHS has a one-dimensional null-space corresponding to the constraint that the net flux of fluid across $\Gamma$ be zero.
We eliminate the null-space by adding
\begin{equation}\label{eq:null-spaceop}
    \int_\Gamma {n}({x})\otimes{n}({y}){\mu}({y})\mathrm{d}y.
\end{equation}
to the RHS of~\eqref{eq:jump} \cite{power, stokessolver, sifuentesrandom}. \jrev{In simply connected domains, this results in a nonsingular operator and a unique solution $\mu(x)$.}

When the domain is multiply connected (i.e.,\ $\Gamma = \Gamma_0 \cup \Gamma_1 \cup ... \cup \Gamma_p$, see Figure~\ref{fig:boundary}), \eqref{eq:stokesbie} cannot represent flows resulting from singularities in the interiors of the holes in the domain (see \cite{degenerate} for further discussion on this). Following \cite{power} the \emph{Stokeslet} and \emph{rotlet} are defined as\footnote{The Stokeslet is the zeroth moment of force, and the rotlet is the antisymmetric part of the first moment of force for the Stokes equations.}
\begin{equation}
S_i\alpha_i \coloneqq \frac{1}{4\pi}\left(\ln \frac{1}{\|{x}-{c}_i\|_2}I +
\frac{
({x}-{c}_i)
\otimes 
({x}-{c}_i)
}{\|{x}-{c}_i\|_2^2}
\right){\alpha}_i, 
\end{equation}
\begin{equation}
    R_i \beta_i \coloneqq \frac{1}{4\pi\|{x}-{c}_i\|_2^2}({x}-{c}_i)^{\perp}\beta_i.
\end{equation}

\begin{figure}[ht!]
  \centering
  \includegraphics[width=0.4\textwidth]{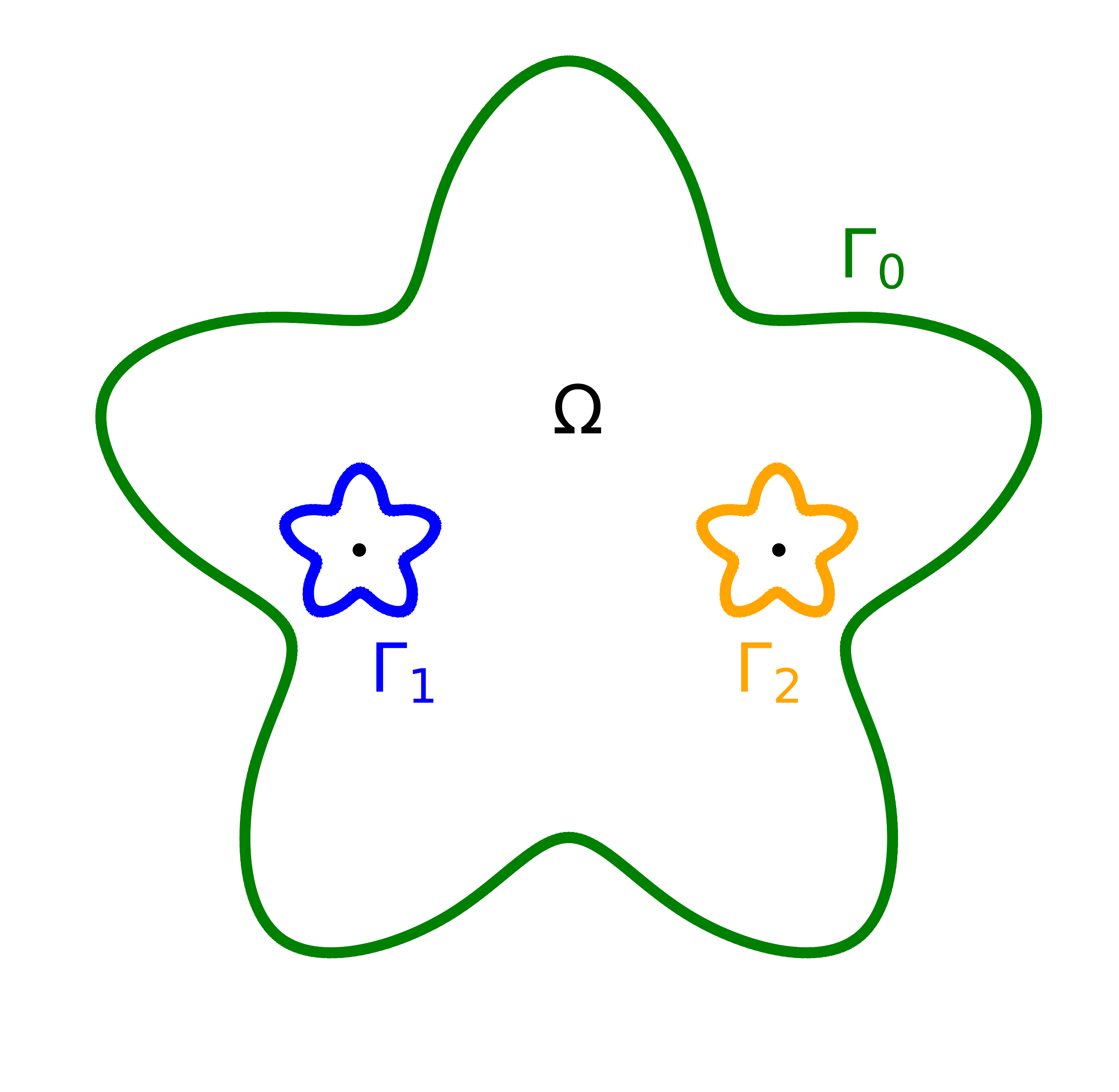}
  \caption{Multiply-connected boundary. The dots inside $\Gamma_1$ and $\Gamma_2$ represent points ($c_1$ and $c_2$ respectively) outside the domain $\Omega$ used in the construction of the Stokeslets and rotlets in Section~\ref{sec:stokeseq}.}\label{fig:boundary}
\end{figure}

The Stokeslet represents the free-space solution of~\eqref{eq:stokes} at $x$ due to a point force at $c_i$, and the rotlet represents the solution at $x$ due to a point torque at $c_i$. \jrev{In the above, $\alpha_i$ and $\beta_i$ represent the strengths of their respective point sources and torques. A Stokeslet of strength $\alpha_i$ located at $c_i$ generates a total force of $\alpha_i$ and a total torque of zero on $\Gamma_i$, and a rotlet of strength $\beta_i$ at $c_i$ generates a total force of zero and a total torque of $\beta_i$ on $\Gamma_i$ (see \cite{power} for further discussion).} Letting $c_1,\dots,c_p$ be \jrev{arbitrary} points on the interiors of $\Gamma_1,\dots,\Gamma_p$ respectively (so that $c_i\notin\Omega$, see dots in Figure~\ref{fig:boundary}), we \jrev{augment our system with the Stokeslets and rotlets, resulting in}
\begin{equation}\label{eq:completed}
     u =  D \mu + S\alpha + R\beta \coloneqq D \mu + \sum_{i=1}^p S_i\alpha_i + \sum_{i=1}^p R_i\beta_i.
\end{equation}\jrev{
\begin{equation}\label{eq:completedback}
    f = -\frac{1}{2} \mu + D \mu + N\mu+ S\alpha + R\beta 
\end{equation}
where $D$ represents the integral operator in~\eqref{eq:stokesbie} and $N$ represents the operator in~\eqref{eq:null-spaceop}. Now} any solution $u(x)$ of~\eqref{eq:stokes} can be represented by~\eqref{eq:completed}\jrev{, and we are tasked with finding $\mu, \alpha,\beta$ which satisfy \eqref{eq:completedback}}.

For conciseness we summarize the equations for $u$ and $f$ as
\begin{equation}\label{eq:forwardeq}
   \begin{bmatrix} u\end{bmatrix} = \begin{bmatrix}
    D & H
    \end{bmatrix}
    \begin{bmatrix}
    \mu \\ \lambda
    \end{bmatrix},
\end{equation}
\begin{equation}\label{eq:missingpsi}
   \begin{bmatrix} f\end{bmatrix} = \begin{bmatrix}
    -\frac{1}{2}I + D+N & H
    \end{bmatrix}
    \begin{bmatrix}
    \mu \\ \lambda
    \end{bmatrix}
\end{equation}
where
\begin{equation}
H \coloneqq  \begin{bmatrix} S & R \end{bmatrix} \qquad \lambda \coloneqq  \begin{bmatrix} \alpha \\ \beta \end{bmatrix}.
\end{equation}
\jrev{For multiply-connected domains, \eqref{eq:missingpsi} is underdetermined since}
\begin{equation}\label{eq:null-space}
    \left(-\frac{1}{2}I + D+N\right)\Psi = {0}
\end{equation}
for the \jrev{null-space}
\begin{equation}
    \Psi = 
    \begin{bmatrix}
    \psi_1^{(1)} & \psi_1^{(2)} & \psi_1^{(3)} & ...& \psi_p^{(1)} & \psi_p^{(2)} & \psi_p^{(3)} 
    \end{bmatrix}
\end{equation}
\begin{equation}
    \psi^{(1)}_i = 
    \begin{cases}
       {e_1}^T & {x}\in \Gamma_i
        \\    \\
        \lbrack
        0 \quad 0
        \rbrack^T & {x}\in \Gamma \setminus \Gamma_i
        
    \end{cases} \qquad    \psi^{(2)}_i = 
    \begin{cases}
      {e_2}^T & {x}\in \Gamma_i
        \\    \\
        \lbrack
        0 \quad 0
        \rbrack^T & {x}\in \Gamma \setminus \Gamma_i
        
    \end{cases}
   \end{equation}
 \begin{equation}
    \psi^{(3)}_i = 
        \begin{cases}
        \left({x}^\perp\right)^T & {x}\in \Gamma_i
        \\    \\
        \lbrack
        0 \quad 0
        \rbrack^T & {x}\in \Gamma \setminus \Gamma_i
    \end{cases}.
\end{equation}
We address this degeneracy by augmenting as \cite{power, stokessolver}
\begin{equation}\label{eq:backwardeq}
   \begin{bmatrix} f \\ 0 \end{bmatrix} = 
   \begin{bmatrix}
    -\frac{1}{2}I + D+N &H \\
    \Psi^T & -I
    \end{bmatrix}
    \begin{bmatrix}
    \mu \\ \lambda 
    \end{bmatrix}.
\end{equation}
Note that we are taking the adjoint of $\Psi$ in the \jrev{operator} sense, and so $\Psi^T\mu$ is an integral. 

The unique solution to \eqref{eq:backwardeq} may be computed and plugged into~\eqref{eq:forwardeq} to find the unique solution to the boundary value problem \eqref{eq:stokes}-\eqref{eq:boundarycondition}.  

\subsection{Skeletonization of a hierarchical submatrix}\label{sec:generalform}
The matrix in~\eqref{eq:backwardeq} is not the discretization of a kernel function satisfying~\eqref{eq:sep}, and hence the tools of Section~\ref{sec:fact} are not directly applicable. To generalize skeletonization to this setting, we factor the top-left block (which \emph{is} the discretization of a kernel satisfying~\eqref{eq:sep}) and apply the resulting interpolation matrices to $H$ and $\Psi^T$ in the following manner
\begin{equation} \label{eq:bigskel}
\begin{bmatrix}
    -\frac{1}{2}I + D +N & H \\
    \Psi^T & -I
    \end{bmatrix}
   \approx \begin{bmatrix}
   U^{-1} & 0 \\ 0 & I 
   \end{bmatrix}
\begin{bmatrix}
    X_D & UH \\
    \Psi^TV & -I
    \end{bmatrix}
   \begin{bmatrix}
   V^{-1} & 0 \\ 0 & I 
   \end{bmatrix}.
\end{equation}
Inverting both sides  of~\eqref{eq:bigskel} and plugging into~\eqref{eq:backwardeq} and subsequently~\eqref{eq:forwardeq} yields 
\begin{equation}\label{eq:solution}
    u \approx\hat{u}=
    \begin{bmatrix}
    D & H
    \end{bmatrix}
  \begin{bmatrix}
   V & 0 \\ 0 & I 
   \end{bmatrix}
\begin{bmatrix}
    X_D & UH \\
    \Psi^TV & -I
    \end{bmatrix}^{-1}
   \begin{bmatrix}
   U & 0 \\ 0 & I 
   \end{bmatrix}
   \begin{bmatrix}
       f \\ 0
   \end{bmatrix}.
\end{equation}

To perform the necessary linear solve, a na\"ive attempt might try and take advantage of the block diagonal structure of $X_D$ and perform block Gaussian elimination with Schur's complement 
\[M_{Schur} = -I - \Psi^TVX_D^{-1}UH.\]
Unfortunately, $ -\frac{1}{2}I + D +N$ is degenerate (see ~\eqref{eq:null-space}) \jrev{and this manifests as poor conditioning of the 
$X_{SS}$ subblock of $X_D$ (recall that we ensured well-conditioning of the $X_{R_iR_i}$ blocks during compression)}, thereby rendering a linear solve with $X_D$ infeasible. To address this, we consider an alternative partitioning of the skeletonized matrix
\begin{equation}\label{eq:partition}
\begin{bmatrix}
X_D & UH \\
\Psi^TV & -I
\end{bmatrix}
= 
\left[\begin{array}{ccc|cc}
X_{R_1R_1} & 0 & 0 &  & \\ 
0 & \ddots & 0 & 0 & \huge (UH)_{R, :}\\
0 & 0 & X_{R_nR_n} &  &  \\
\hline
 & 0 &  & X_{SS} & (UH)_{S,:}\\ 
 & (\Psi^TV)_{:,R} &  & (\Psi^TV)_{:,S} & -I 
\end{array}
\right].
\end{equation} 
In this case the diagonal blocks are nonsingular, and we may safely use block Gaussian elimination to solve the linear system. \jrev{The \secondrev{time} complexity of this new factorization is given by Corollary~\ref{cor:cost}.
\begin{corollary}\label{cor:cost}
Assuming $p\ll N$ and $k_l$ \secondrev{is bounded as in \eqref{eq:kl}}, the cost $t_f$ of computing a recursive skeletonization factorization of the augmented system~\eqref{eq:backwardeq} is 
\begin{equation}
t_f = 
        \begin{cases}
        \mathcal{O}(N) & d=1 \\
        \mathcal{O}(N^{3(1-1/d)}) & d>1
    \end{cases},
\end{equation}
 and the cost of then applying the factorization to a vector or performing a linear solve is
 \begin{equation}\label{eq:rsmcascost}
t_{a/s} = 
        \begin{cases}
        \mathcal{O}(N) & d=1 \\
        \mathcal{O}(N\log{N}) & d=2\\
        \mathcal{O}(N^{2(1-1/d)}) & d>2
    \end{cases}.
\end{equation}
\end{corollary}
\begin{proof}
The cost of the factorization is dominated by the cost of factoring the left block, the complexity of which is given by Theorem~\ref{thm:rscost}. The computation of the $\mathcal{O}(|S|)\times \mathcal{O}(|S|)$ Schur's complement for the lower right block has the lower order cost $\mathcal{O}(k_0^3p^3) + t_{s}\mathcal{O}(p).$
\end{proof}}

We also note that the updating scheme from Section~\ref{sec:factupdate} naturally generalizes to solving~\eqref{eq:solution} using the partitioning in~\eqref{eq:partition}. As before, we avoid unnecessarily recomputing and refactoring subblocks of $X_D$, $U$, and $V$ corresponding to boxes far from the perturbation. Further, interior holes can be added or deleted with ease when using an adaptive tree. 

\begin{remark}
\jrev{
If HIF-IE recompression techniques are used, then the complexity is the same as in Theorem~\ref{thm:hifie}. Similarly, because the dominant cost is that of factoring the upper left block of~\ref{eq:backwardeq} the cost of updating the factorization is given by Lemma~\ref{lem:updcost}. 
}
\end{remark}

\section{Numerical results}\label{sec:results}
We implemented the factorization routine in C++ using \jrev{BLAS and LAPACK} for matrix operations and OpenMP for parallelization. The code and experiments are available at \url{https://github.com/jpryan1/kern-interp}. All \secondrev{serial} 2D testing was conducted on a workstation with a 3.6 GHz quad-core Intel Core i7 CPU and 32 GB of RAM, and \secondrev{parallel and 3D testing were conducted on an Amazon Web Services compute-optimized c5d.18xlarge instance with 72 vCPUs and 154.6 GB of RAM. Tests that didn't involve scaling threads were conducted in serial.} All timing plots display average times over three runs of the relevant computation. 

\jrev{
For our 2D experiments, we used a non-uniform trapezoid quadrature for integration. The non-uniformity comes because we define the boundaries by placement of spline knots for cubic spline interpolation, and then choose points along the splines (uniformly in the spline parameter) as our integration nodes. For our 3D experiment, we triangulate the domain, using the triangle barycenters as integration nodes and triangle surface areas as weights. }
\jrev{
\subsection{Accuracy}
To verify the accuracy of the solver, we ran experiments to examine the error of the computed solution \secondrev{to the Stokes equations} when compared to (a) the solution computed via dense linear algebra and (b) the analytic solution, when available. \secondrev{For (a) we use the problem setup in Section~\ref{sec:optfluid}, and for (b) we use a boundary of concentric circles for which analytic solutions to the Stokes equations are available.} Results from these trials are visualized in Figure~\ref{fig:linalg_acc}. In both cases the overall accuracy of the solver tracks the expected accuracy based on the tolerance used for the interpolative decompositions\textemdash in additional experiments not reported here we observe this for the three dimensional Stokes flow problem as well.
}

\jrev{In all cases below, when the factorization is updated the root node maintains the same position and size and we perform dynamic tree maintenance as alluded to in Section~\ref{sec:factupdate}. As a result, we achieve \emph{exactly the same factorization} from updating as we do from recomputing the factorization. This has been tested and confirmed in all updating experiments. This behavior is in agreement with the initial work on updating rank-structured factorizations~\cite{mindenupdate}.} 


\begin{figure}[ht]
  \centering
  \includegraphics[width=0.45\textwidth]{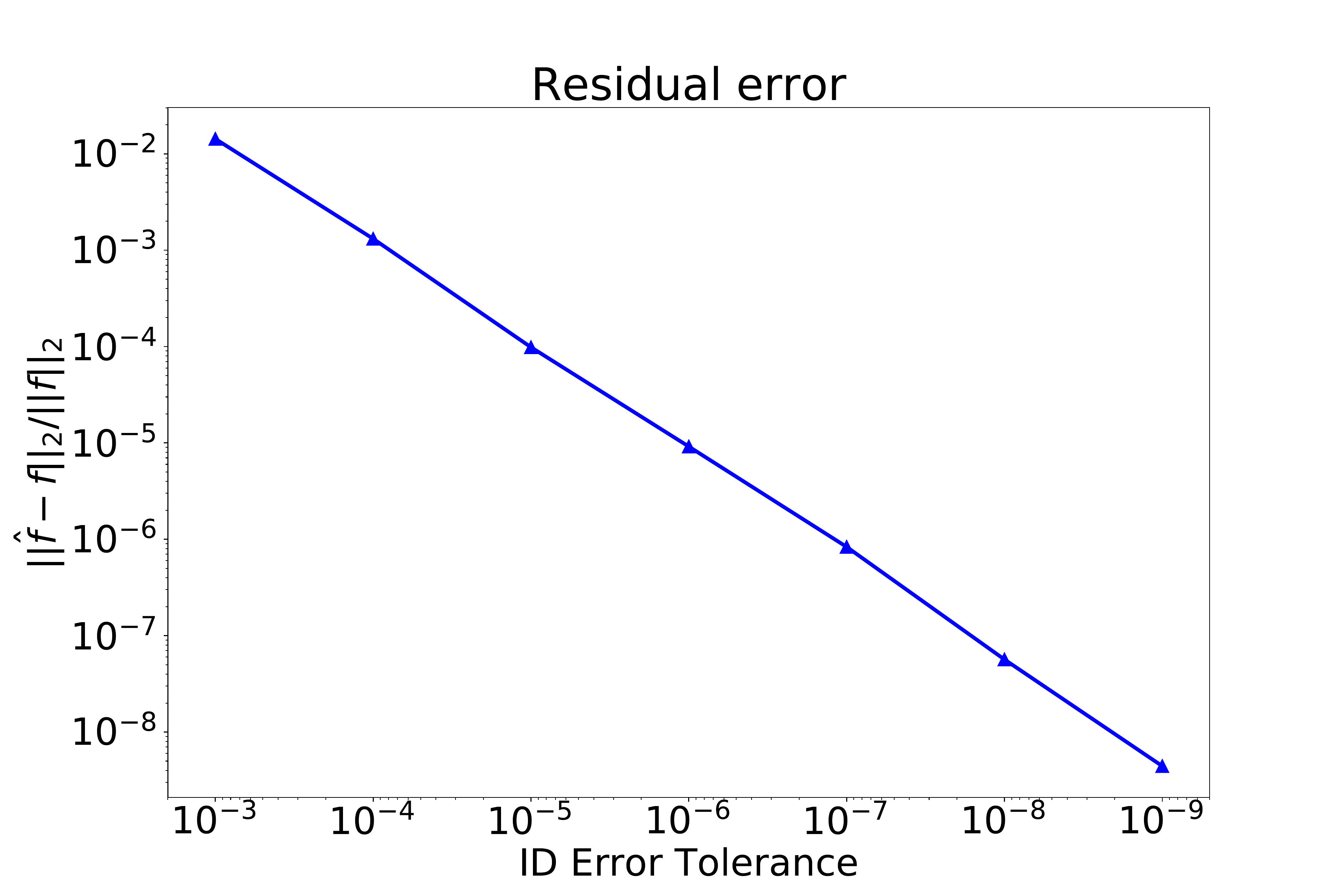}
    \includegraphics[width=0.45\textwidth]{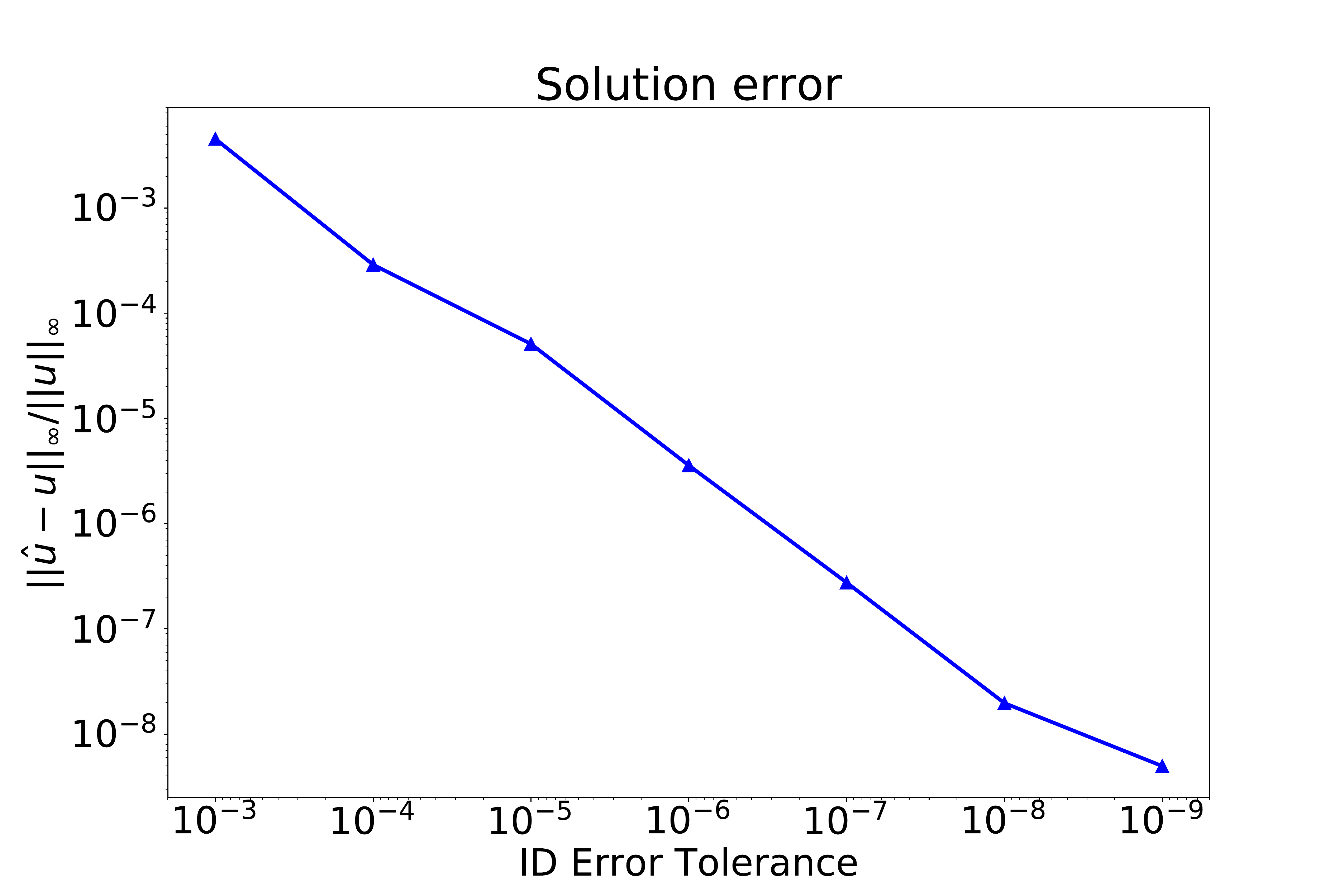}
  \caption{\jrev{\secondrev{Left: r}esidual between the true boundary data $f$ and $\hat{f}=K\hat{u},$ where $\hat{u}$ is the solution computed to $f = Ku$ by the rank-structured factorization with varying tolerance in the interpolative decomposition. This plot uses the problem setting from Figure~\ref{fig:demo3b}.} \secondrev{Right: relative error between the true solution and computed solution in a 2D problem where the analytic solution is available. In this experiment, the boundary used is two concentric circles with tangential boundary velocities, the number of boundary points is set at 32,768, and the solution is evaluated at points generated by laying down a 200x200 grid and selecting those inside the domain. We have omitted results from the analogous 3D experiment with concentric spheres as it provided similar results.}
  }\label{fig:linalg_acc}
\end{figure}

\begin{figure}[ht!]
  \centering
  \includegraphics[width=0.4\textwidth]{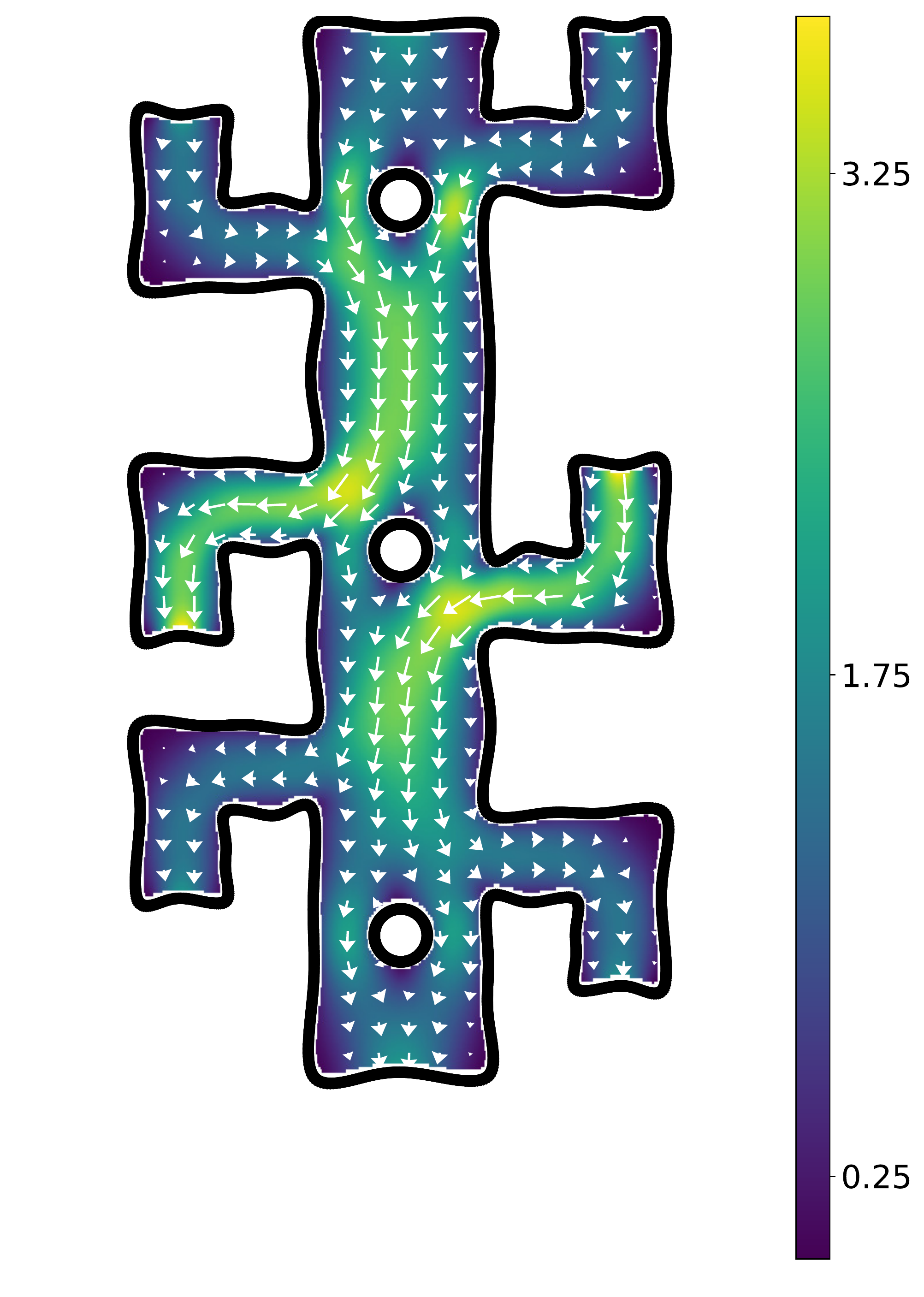} \includegraphics[width=0.9\textwidth]{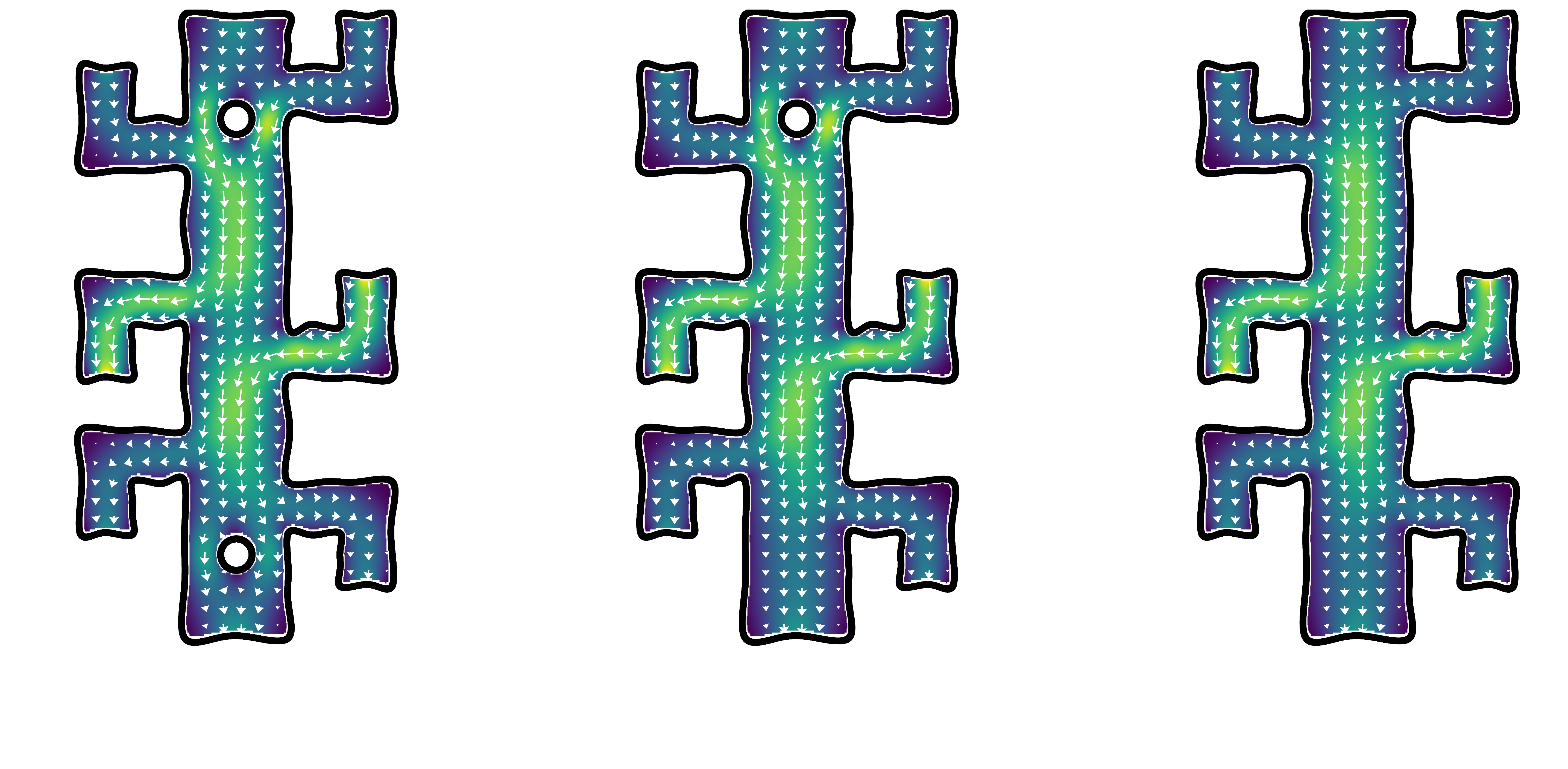}
  \caption{\jrev{Stokesian flow velocities in a multiply-connected domain. The solution is quickly recalculated after modifying the inner holes. For this figure we discretize the boundary using \secondrev{32,768} integration nodes, with 3/4 of them on the outer boundary and the rest evenly distributed across the interior holes, rounded as appropriate. The solution is evaluated throughout the domain, though we intentionally avoid evaluation near the boundaries as more sophisticated techniques are required \cite{qbx}.} }\label{fig:demo1}
\end{figure}
\begin{figure}[ht]
  \centering
  \includegraphics[scale=0.27]{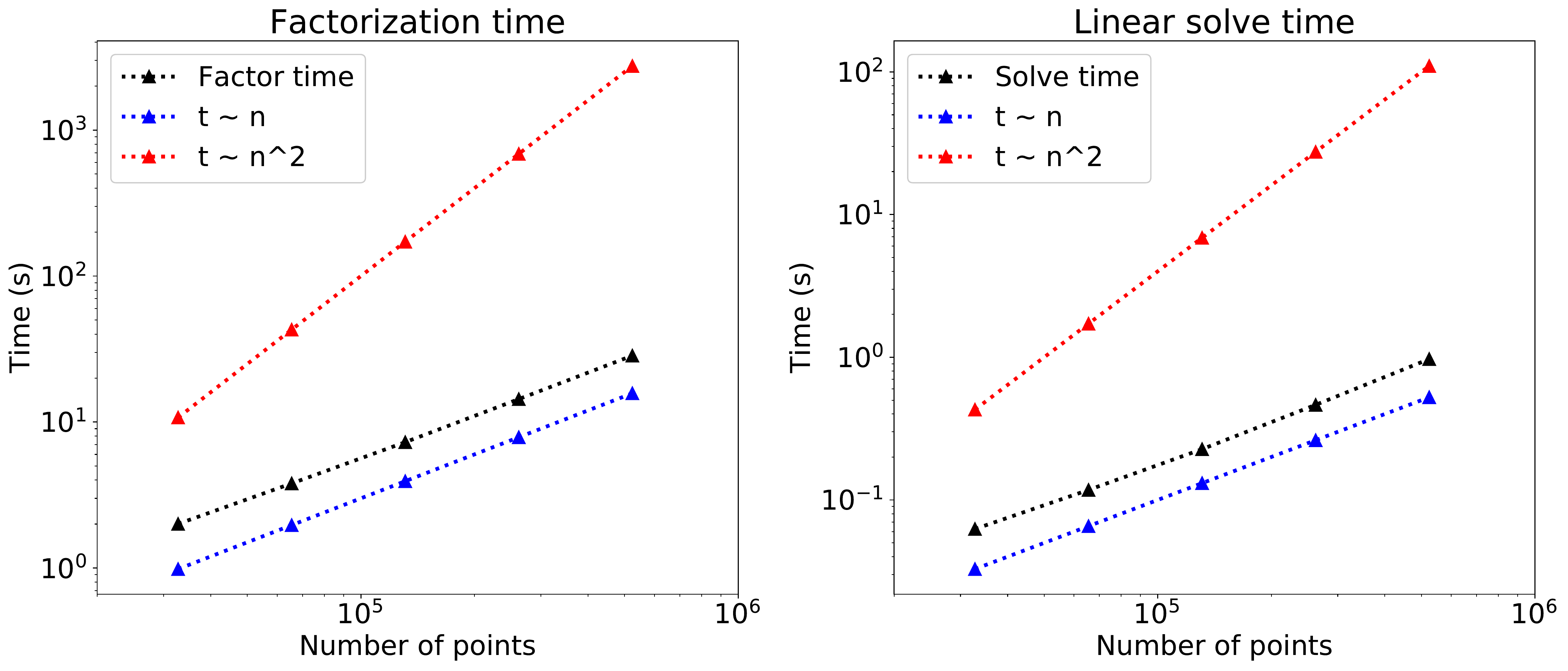}
  \caption{Both factoring and solving demonstrate linear scaling in the number of discretization points along the boundary, in accordance with Corollary~\ref{cor:cost}}\label{fig:scaling}
\end{figure}

\jrev{
\subsection{Stokes flow with addition/deletion of interior holes}\label{sec:ex1}
In our first experiment (visualized in Figure~\ref{fig:demo1}), we create a 2D outer boundary of a channel with intake and outtake pipes via cubic spline interpolation of 121 prescribed spline knots (note that this enforces smoothness of the boundary, and manifests as a slight curvature of the boundary in Figure~\ref{fig:demo1}). On the inside are three circular interior holes. For boundary conditions, we assign the following Dirichlet data:
\begin{itemize}
    \item Let $\mathcal{S}_i$ denote the set of points on the curve at the entrance/exit of the $i$th pipe. Let $m_i=\min_{x\in\mathcal{S}_i}{x_1}$ and $M_i=\max_{x\in\mathcal{S}_i}{x_1}$ be the minimum and maximum values of the horizontal components in these sets. Then we set
    \[f(x)=\left(0, \cos{\left(2\pi \frac{x-m_i}{M_i}\right)}-1\right) \]
    except for the pipes in the middle, where we have \[f(x)=\left(0, 2\cos{\left(2\pi \frac{x-m_i}{M_i}\right)}-2\right). \]
    (Notice in Fig.~\ref{fig:demo1} that flow is stronger in the middle pipes.)
    \item Everywhere else we set $f(x)=0$.
\end{itemize}}

\jrev{The above Dirichlet data automatically satisfies the consistency condition coming from incompressibility. Figure~\ref{fig:scaling} shows that factoring and solving the associated linear system scale linearly with the number of quadrature nodes on the boundary, in accordance with Corollary~\ref{cor:cost}. Note in Figures~\ref{fig:scaling} and~\ref{fig:ex1speedup} $N$ points with $p$ boundaries in a Stokes flow experiment result in a matrix of size $2N+3p$ (see~\eqref{eq:backwardeq}).}

\begin{figure}[ht]
  \centering
  \includegraphics[scale=0.23]{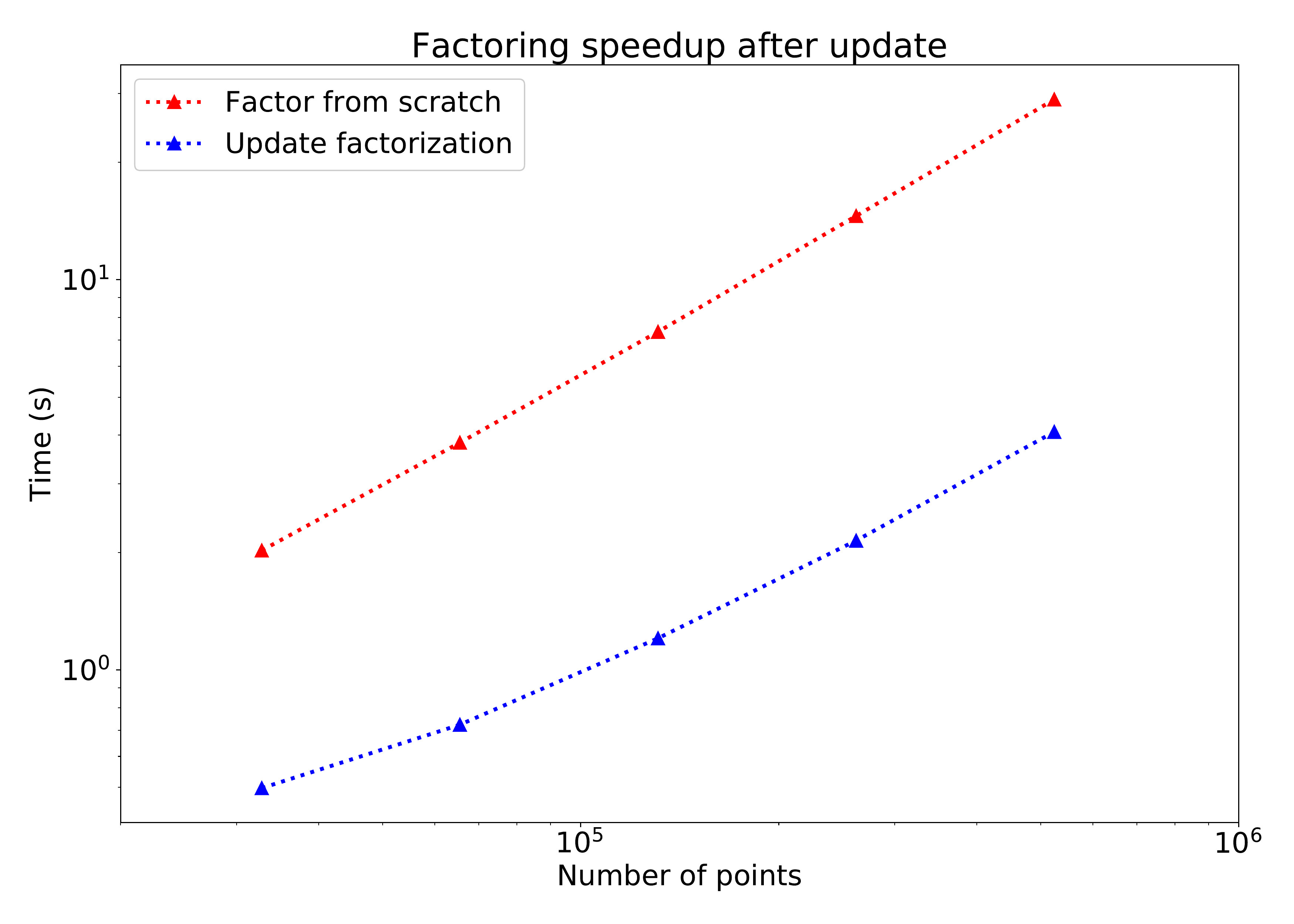}
  \caption{The updating scheme of Section~\ref{sec:factupdate} results in a 4-8x speedup in the factorization of the linear system \secondrev{described in Section~\ref{sec:ex1}}. \secondrev{Plotted is the speedup following insertion of an interior hole achieved by updating the factorization versus recomputing the factorization from scratch.} } \label{fig:ex1speedup}
\end{figure}

\jrev{
After the initial factorization, \secondrev{we modify the boundary by either adding or deleting interior holes}, and the factorization is updated based on these perturbations. Each update corresponds to \secondrev{modifying} $N/12$ points \secondrev{on} the boundary, where $N$ is the initial total number of points.} In Figure~\ref{fig:ex1speedup} we see that updating the factorization based on the technique in Section~\ref{sec:factupdate} results in a substantial speedup over recomputing the factorization from scratch. \secondrev{T}he cost of computing a linear solve in the perturbed geometry is effectively the same regardless of whether that factorization is from scratch or from an update.

\jrev{Parallelizing as described in Section~\ref{sec:parallel} yields notable speedups. Figure~\ref{fig:parplot} shows that a \secondrev{4-5x} and 2x speedup is achieved in time and linear solve time respectively by using \secondrev{eight} threads. Using more than \secondrev{eight} threads does not appear to have a sizeable impact, \secondrev{ostensibly due to the ratio of work near the top of the tree and the concentration of points within 4 and 8 tree nodes at the second and third levels of the tree. In Section~\ref{sec:cowupd} we see greater gains from parallelizing a 3D problem where points are more spread throughout tree nodes at the third level of the tree.} }

\begin{figure}[ht]
  \centering
  \includegraphics[scale=0.27]{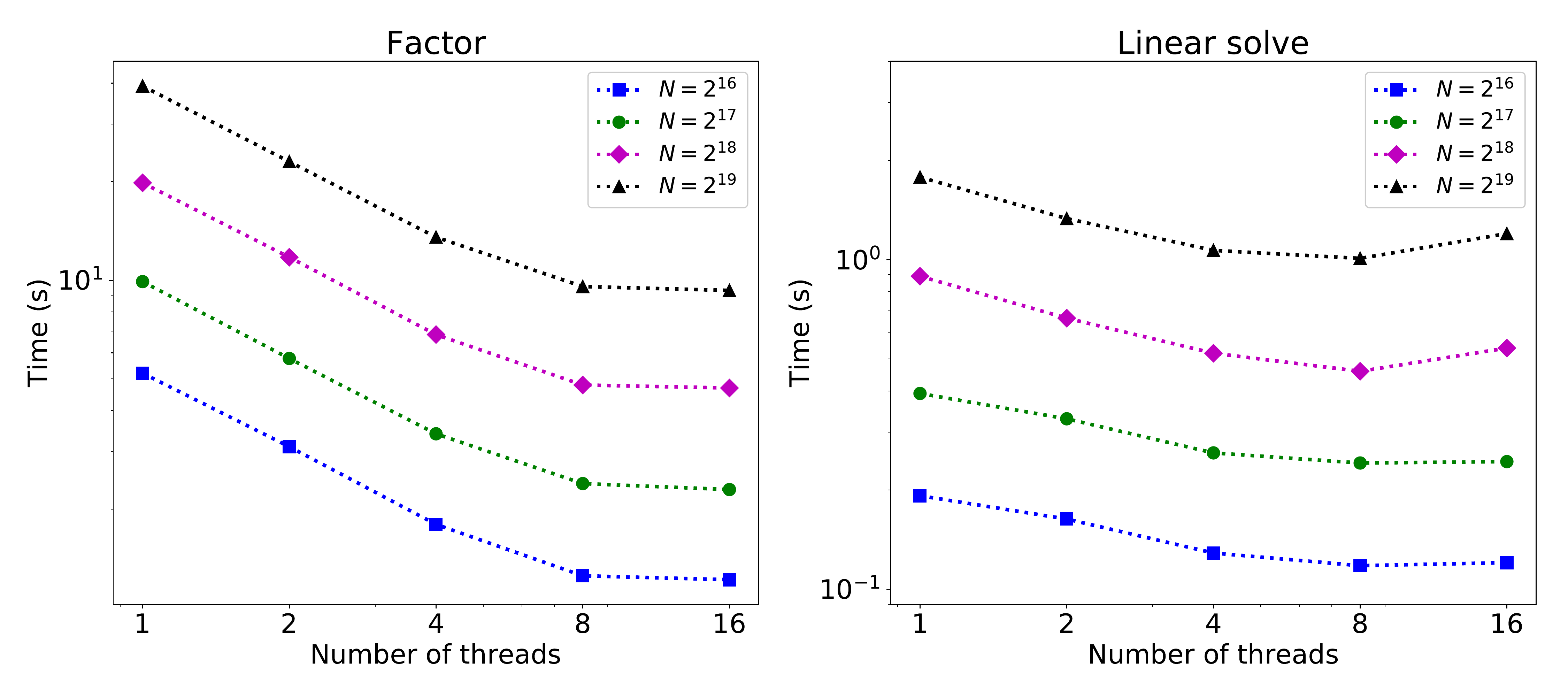}
  \caption{\jrev{Parallelizing the initial factorization and linear solve for \secondrev{the system described in Section~\ref{sec:ex1}} results in notable speedups up to \secondrev{eight} threads, with not much improvement seen by then increasing to \secondrev{sixteen} threads. The speedup using \secondrev{eight} threads is a factor of about \secondrev{4-5} for the factorization, and around 2 for the solve.}}\label{fig:parplot}
\end{figure}

\subsection{Stokes flow through a channel with many moving interior holes}
In our second experiment (zoomed-in visualization in Figure~\ref{fig:demo2}), we simulate Stokes flow in a channel with ten interior circular holes. The outer boundary is a long rectangle with rounded corners, and the boundary conditions are again Dirichlet data. \jrev{We prescribe horizontal velocities on the left and right walls whose magnitudes vary as $\cos{(x)}$, just as in the previous experiment. On the top and bottom walls and on the interior circles, we prescribe zero velocities (no slip).}

\begin{figure}[ht]
  \centering
  \includegraphics[scale=0.35]{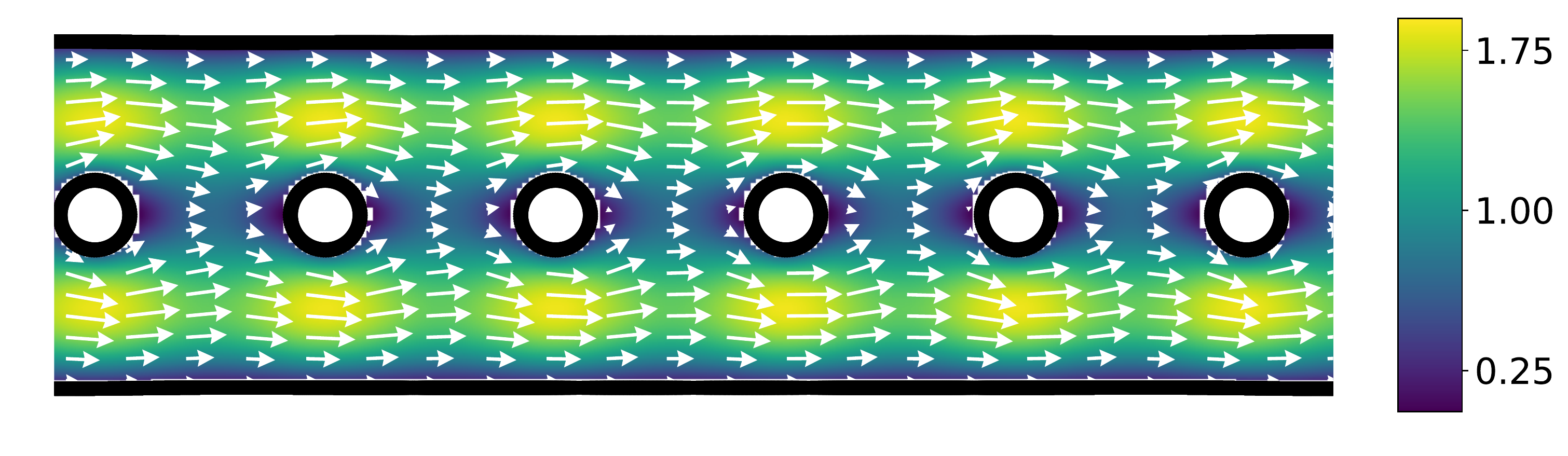}
  \includegraphics[scale=0.2]{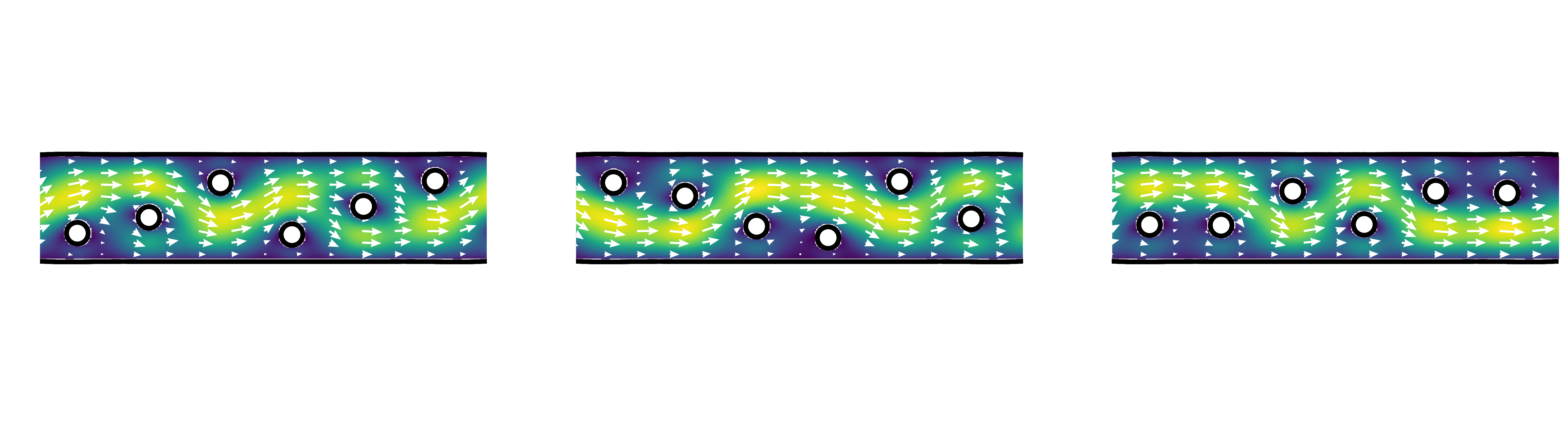}
   \caption{Stokes flow solution for a channel-like outer boundary with interior holes. We initially construct a factorization for the top geometry, and use the updating scheme to develop factorizations for the bottom geometries. We discretize the boundary using \secondrev{32,768} integration nodes, with 3/4 of the integration nodes on the outer boundary and 1/32 on each of the eight interior boundaries, rounded as appropriate. The above image is zoomed in to show the flow in detail.}\label{fig:demo2}
 \end{figure}

To illustrate the power of maintaining the same factorization structure across updates, we perform one hundred changes to the positions of the interior holes. \jrev{Each of these changes corresponds to the repositioning of $N/32$ points. Importantly, we are assuming the solution is desired for each of the 100 configurations and are not using a sequence of updates to get to a single large update---in that case we would simply do a single update with all the necessary changes.} Table~\ref{table:times} shows that, after initial factorization, subsequent factorizations (for different problem geometries) require significantly less time to compute by using the updating scheme. Furthermore, the time required to update the factorization remains relatively stable.

\jrev{
\begin{table}[ht]
\caption{\label{table:times}Cost of initial factorization and subsequent updates, along with statistics for factorization update times. For this experiment, we discretize the boundary using 131,072 integration nodes.}
\centering
\begin{tabular}{|c|c|c|c| c|} 
 \hline
Initial fact.  & 1st update  & 100th update  &  Update time $\mu$ &  Update time $\sigma$ \\
\hline
7.40 s & 0.86 s & 0.85 s & 0.86 s & 0.01 s\\
 \hline
\end{tabular}
\end{table}
}


As described at the end of Section~\ref{sec:factupdate}, the speedup due to parallel compression of boxes on each level is mostly seen in the initial factorization, with the factorization updates benefiting relatively less. One way to maintain efficient processor usage is to compute the initial factorization in parallel as described in Section~\ref{sec:parallel}, and then compute distinct updates each on independent processors. This is particularly beneficial for problems in which we know a priori a large number of related geometries in which we would like to know the solution. An example of this is exploring a solution landscape locally in an optimization problem. We explore this setting in the following experiments.

\subsection{Shape optimization}\label{sec:optex}
\subsubsection{Optimizing heat source/sink placement}\label{sec:ex3a}
In this experiment (visualized in Figure~\ref{fig:demo3a}) we calculate steady-state temperature distributions given Neumann boundary conditions, i.e., we solve Laplace's equation
\[\Delta u({x}) = 0 \qquad {x}\in\Omega \]
for many geometries. Our goal will be to maximize an objective function of the solution to be discussed later. The setup is similar to the Stokes problems, with the following exceptions:
\begin{itemize}
\item Equation~\eqref{eq:forwardeq} becomes 
    \[u = S\mu\]
    where $S$ is the \emph{single layer potential} for the 2D Laplace problem defined as 
    \[S\mu \coloneqq -\int_{\Gamma} \frac{1}{2\pi}
    \log{\left|{x}-{y}\right|}{\mu}({y})\mathrm{d}y .\]
\item Equation~\eqref{eq:backwardeq} becomes 
    \[ f=\left(-\frac{1}{2}I - D + N\right)\mu. \]
Notably, the integral operator is full rank in this setting, and we do not have to augment the system. 
\item The double layer potential is now
    \[D\mu \coloneqq -\int_{\Gamma} \frac{1}{2\pi}\frac{({x}-{y})\cdot {n}({x})}{({x}-{y})\cdot ({x}-{y})}{\mu}({y})\mathrm{d}y.\]
\end{itemize}

\begin{figure}[ht]
  \centering
  \includegraphics[width=0.5\textwidth]{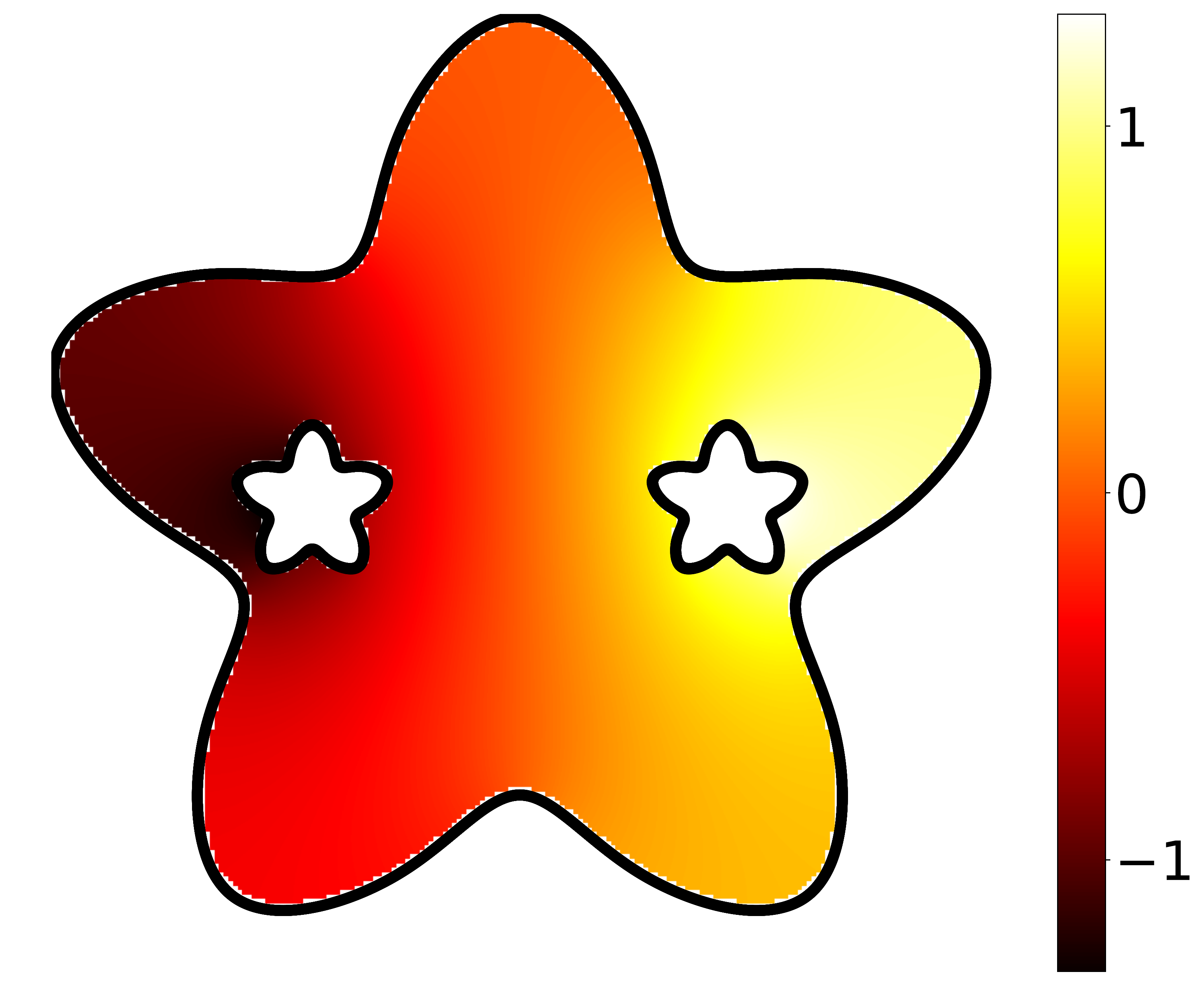}
   \caption{Steady state heat distribution given flux conditions on the boundaries. We discretize the boundary using \secondrev{16,384} integration nodes, with 2/3 of the integration nodes on the outer boundary and 1/6 on each of the two interior boundaries, rounded as appropriate. Visualized is the geometry with $\theta_1=0$ and $\theta_2=\pi$.}\label{fig:demo3a}
 \end{figure}

We create a starfish outer boundary via cubic spline interpolation with 20 prescribed spline knots. For the interior holes, we use starfishes that are 16\% the size of the outer boundary. \jrev{For the Neumann data we set
\begin{equation}\label{eq:neumanndata}
    \frac{\partial u}{\partial n_x} =
    \begin{cases} 
    0 & x\in \Gamma_0\\
    1 & x\in \Gamma_1\\
     -1 & x\in \Gamma_2
     \end{cases}
\end{equation}
where the LHS is the normal derivative at the point $x$ on the boundary. 
}
We consider positions of the interior holes that lie along a path which is equal to a scaled-down version of the outer boundary. As each hole moves along this path, it is parametrized by the periodic variable $\theta_i\in[0,2\pi)$. To prevent the holes from overlapping, we require that the distance between $\theta_1$ and $\theta_2$ (modulo $2\pi$) be greater than $\pi/4$. 

For the objective function, we estimate the derivative of the solution in the x-direction at the center of the domain, and try to find $(\theta_1, \theta_2)$ which maximizes its value (see Figure~\ref{fig:landscape} for a visualization of this function). We choose updates to $(\theta_1,\theta_2)$ via gradient descent, where the gradient of the objective function is estimated by the fourth-order centered finite difference approximation. In other settings, more sophisticated techniques of analytically or numerically computing the gradients may be more appropriate, see \cite{adjoint, adjointaero}. We choose the length of our descent step via a backtracking line-search. Hence, at each optimization step, we require the solution at a minimum of eight \jrev{(fourth-order finite difference for two parameters)} distinct but closely related geometries.

\begin{figure}[ht]
  \centering
  \includegraphics[width=0.7\textwidth]{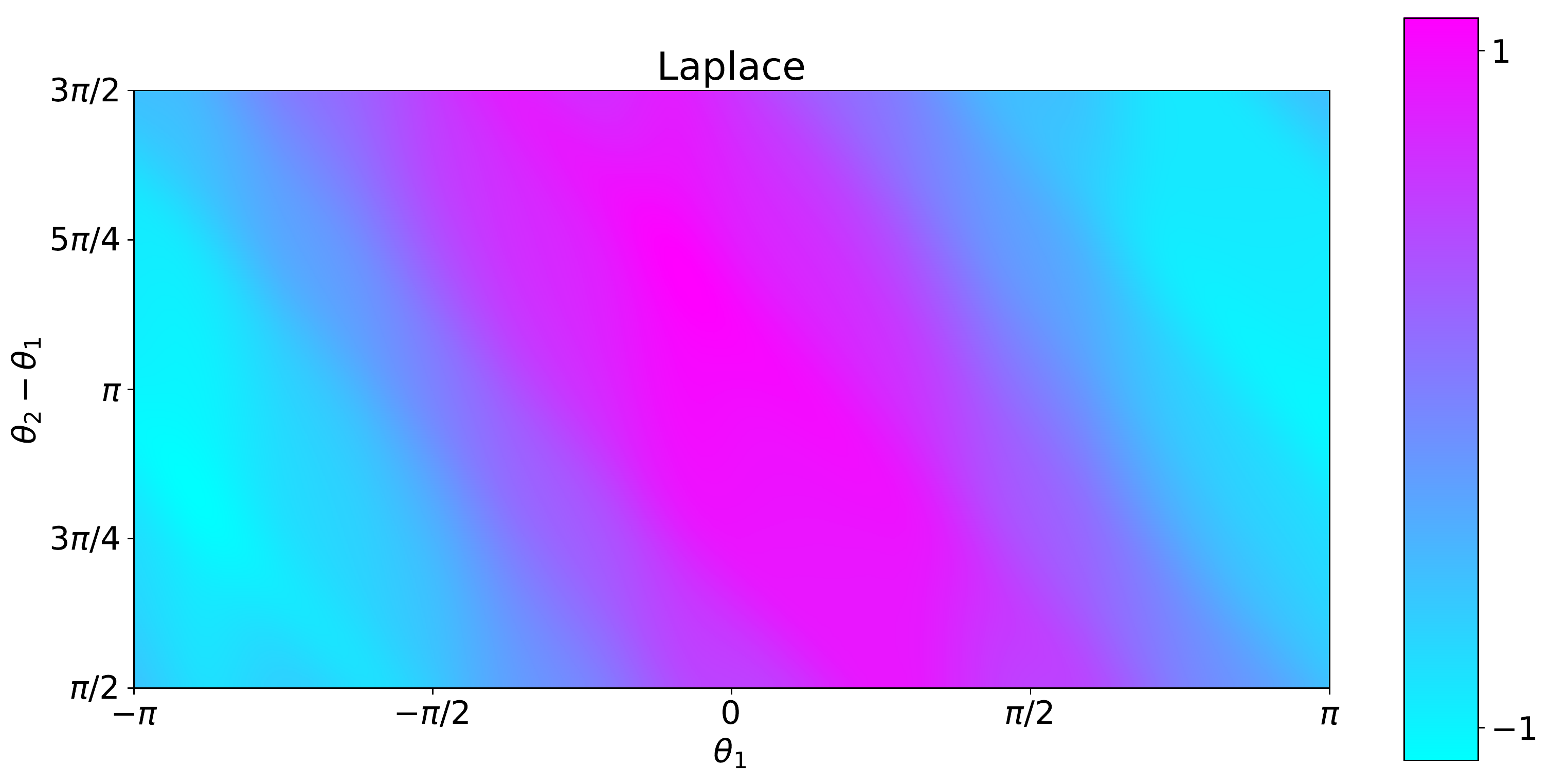}\\
  \hspace{3.3mm}\includegraphics[width=0.715\textwidth]{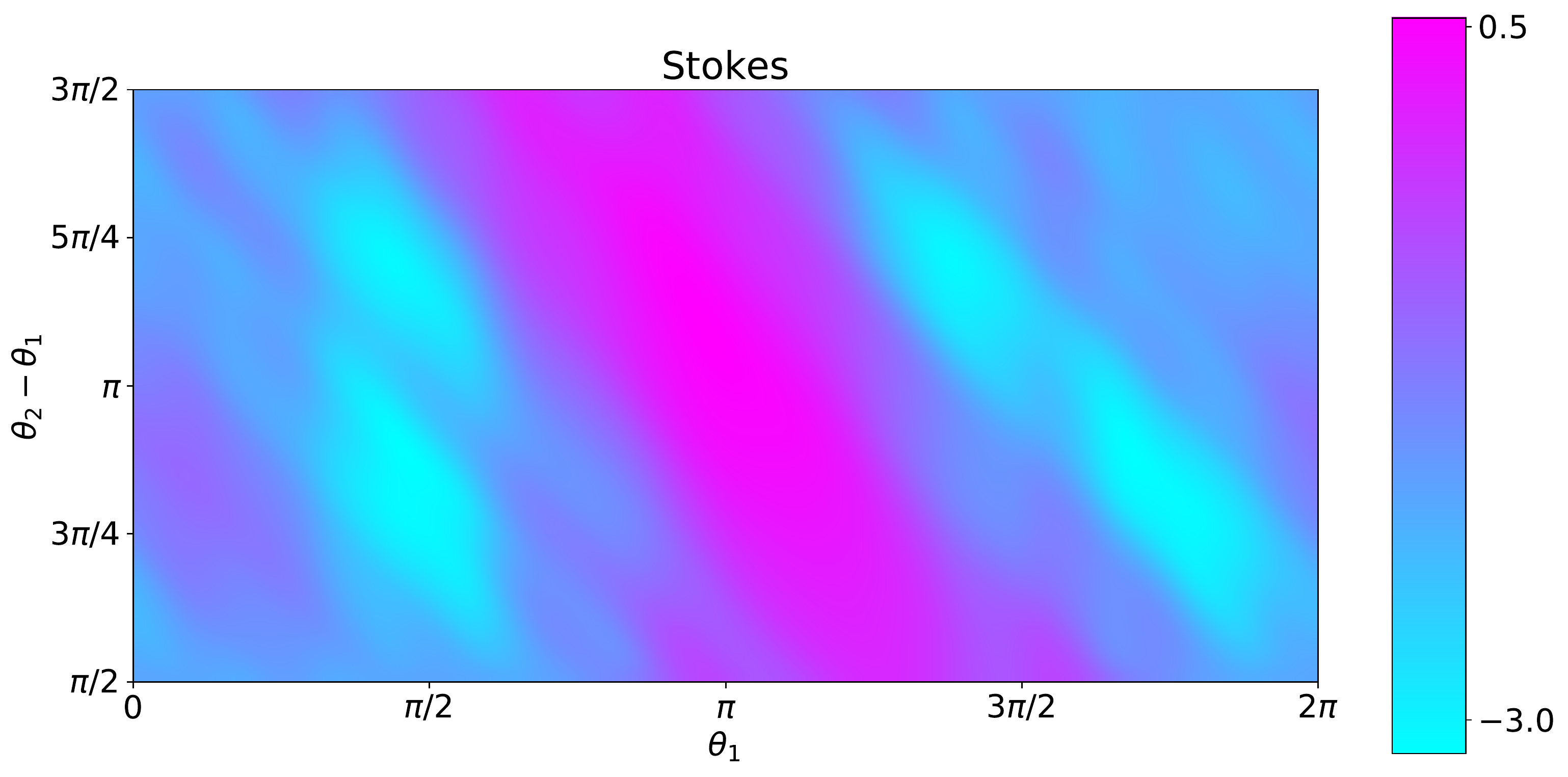}
  \caption{Objective function value for 2,500 inner hole parameter pairs $(\theta_1, \theta_2)$. Note that while each plot seems to exhibit a good local maximum near the center, the landscape for Stokes flow is more complex, containing many local maxima. For these plots, we discretize the boundary using \secondrev{16,384} integration nodes. }\label{fig:landscape}
\end{figure}

As discussed in Section~\ref{sec:factupdate}, the factorization updates are less parallelizeable than the initial factorization. \jrev{Although every level (besides the root) will contain multiple boxes in need of recomputation, the number of such boxes will be considerably smaller than in the initial factorization, and using multiple processors will result in a relatively higher synchronization cost.} Instead of using multiple processors \secondrev{for each} factorization update, we can take advantage of them by computing the \jrev{gradient approximation} updates in parallel, \secondrev{using fewer threads per update (see Table~\ref{table:schemes})}. As a result, the performance gain from using the updating routine in an optimization setting is even greater than in Table~\ref{table:times} on a per-step measure (see Figure~\ref{fig:parallel}).

\begin{figure}[ht]
  \centering
  \includegraphics[scale=0.25]{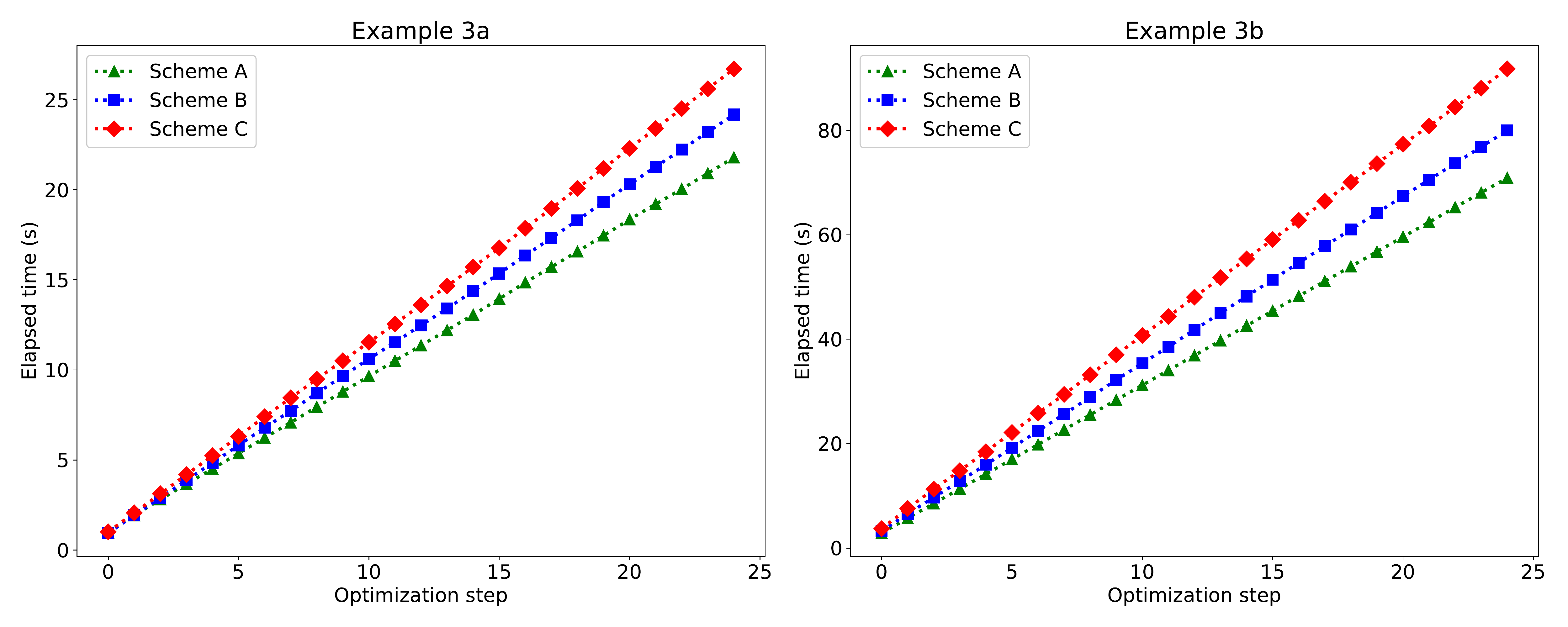}
  \caption{Elapsed times of the optimization experiment described in Section~\ref{sec:ex3a} when using three different schemes for allocating work amongst four processors. In Scheme A, the related factorizations in the finite difference gradient approximation are each computed by a single processor. In Scheme B, each factorization is computed by two processors working in parallel, hence at most two updates are computed simultaneously. In Scheme C, each factorization is computed by four processors working in parallel, hence no updates are computed simultaneously. \secondrev{The parallelization schemes are summarized in Table~\ref{table:schemes}.} We discretize the boundary using \secondrev{16,384} integration nodes.}\label{fig:parallel}
\end{figure}


\begin{table}[ht]
\secondrev{
\caption{\label{table:schemes}Parallelization schemes used in the optimization experiments with 4 processors.}}
\centering
\begin{tabular}{lccc} 
Scheme& A & B & C \\
\hline
Processors per update &1&2&4\\
Simultaneous updates &4&2&1\\
\end{tabular}
\end{table}

Initializing the interior holes to be relatively far from their seemingly optimal positions, we see (Figure~\ref{fig:convergence}, left) rapid convergence to a local maximum. The value of $(\theta_1,\theta_2)$ at the computed maximum matches that of the maximum seen in the center of Figure~\ref{fig:landscape}, top, and aligns with the intuitive expectation that the horizontal temperature gradient at the center of the domain is maximized by placing the heat source and sink on the left and right of the center. 

\begin{figure}[ht]
  \centering
  \includegraphics[scale=0.27]{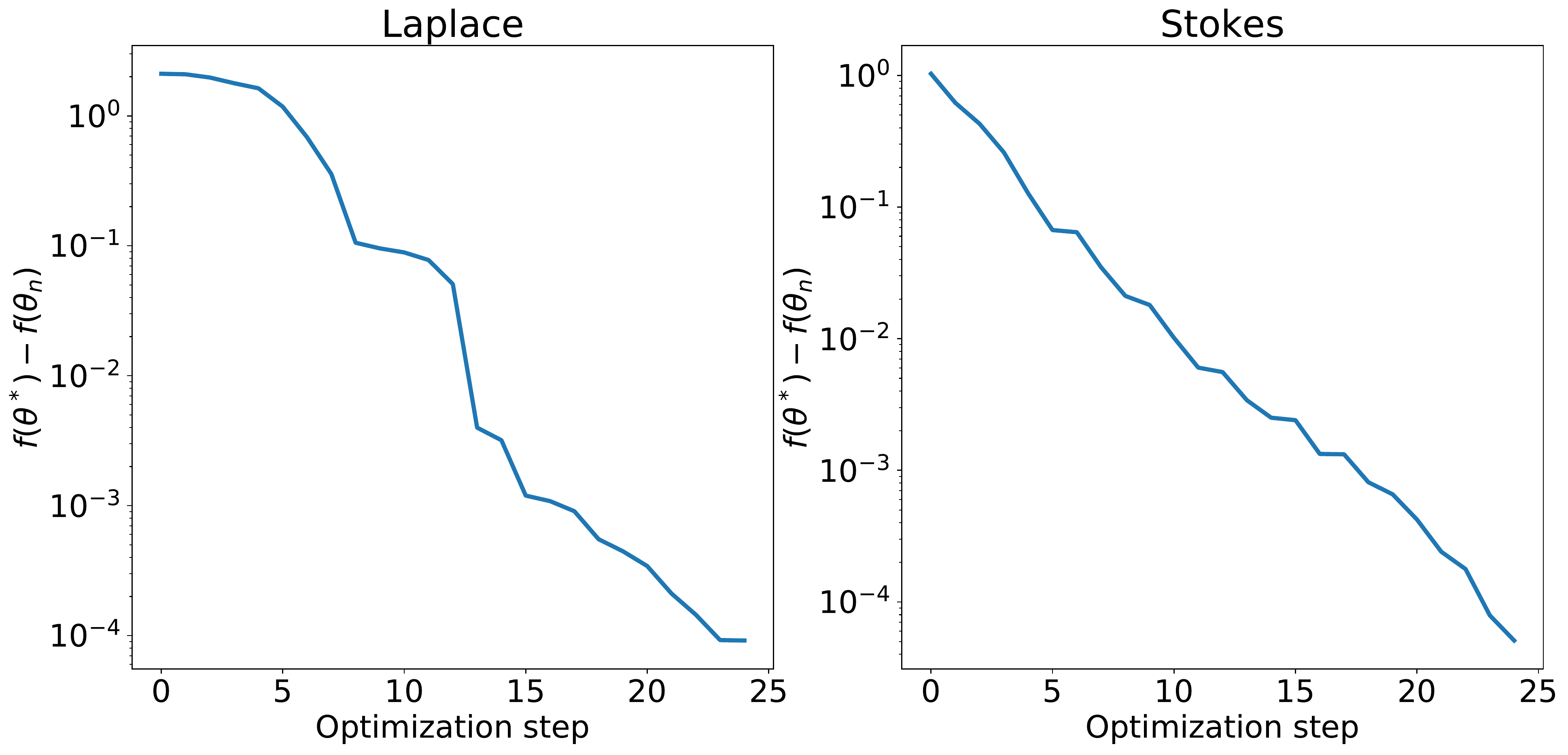}
  \caption{Convergence to optimal configurations. The parameters found to maximize the objective functions match those predicted by Figure~\ref{fig:landscape}. In the above convergence plots, the true optimal value $f(\theta^*)$ is approximated by allowing the optimization to run for a long time from an initial point based on Figure~\ref{fig:landscape}. As in Figure~\ref{fig:parallel}, we discretize the boundary using \secondrev{16,384} integration nodes. }\label{fig:convergence}
\end{figure}

\subsubsection{Optimizing fluid source/sink placement}\label{sec:optfluid}
In this experiment (visualized in Figure~\ref{fig:demo3b}), we return to the Stokes equations and reuse the boundary from the previous experiment. \jrev{We use the following Dirichlet data as boundary conditions
\begin{equation}\label{eq:ex3bdata}
f(x) = \begin{cases}
(1,0) & x\in \Gamma_0 \\
n_x & x\in\Gamma_1\\
-n_x & x\in \Gamma_2
\end{cases}
\end{equation}
where $n_x$ is the unit normal vector at the point $x$ on the boundary. 
}
As our objective, we aim to maximize the horizontal flow to the left at the center of the domain. This choice yields interesting effects, as illustrated in Figure~\ref{fig:landscape}, bottom\textemdash besides acting as fluid sources and sinks, the interior holes serve to obstruct the horizontal flow coming from the outer boundary. As a result, we see greater complexity in the dependence of the objective function on $(\theta_1,\theta_2)$. As in the previous experiment, we see performance gains by efficient allocation of work among the four processors (Figure~\ref{fig:parallel}, right) and rapid convergence to a local maximum (Figure~\ref{fig:convergence}, right).

 \begin{figure}[t]
  \centering
  \includegraphics[width=0.5\textwidth]{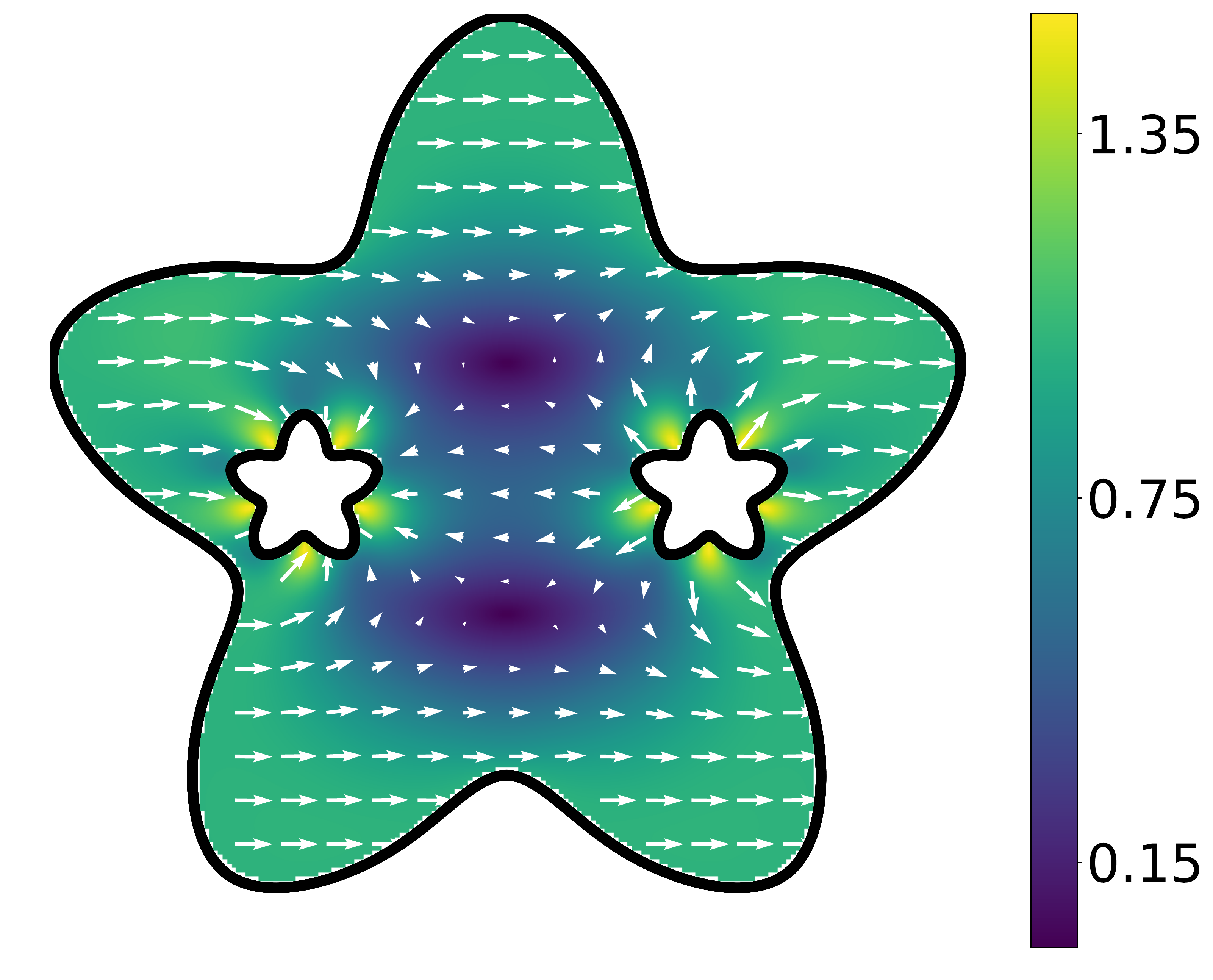}
   \caption{Stokes flow inside the domain given prescribed velocities along the boundary. We discretize the boundary using \secondrev{16,384} integration nodes, with 2/3 of the integration nodes on the outer boundary and 1/6 on each of the two interior boundaries, rounded as appropriate. Visualized is the geometry with $\theta_1=0$ and $\theta_2=\pi$.}\label{fig:demo3b}
 \end{figure}

\jrev{
\subsection{Stokes flow in a \secondrev{3D domain with a moving interior hole}}\label{sec:cowupd}
In our third experiment (visualization in Figures~\ref{fig:cow_cross}, \ref{fig:cow}), we simulate Stokes flow \secondrev{in a 3D domain with a spherical outer boundary, a spherical interior hole, and a cow-shaped interior hole.} The double layer potential is now
\begin{equation}\label{eq:dl3d}
D\mu \coloneqq -\int_{\Gamma}\frac{3}{4\pi}\frac{(x-y)\cdot n(y)}{\|x-y\|_2^5}(x-y)\otimes(x-y)\mu(y)\mathrm{d}y.
\end{equation}
The Stokeslets and rotlets are 
\begin{equation}
S_i\alpha_i \coloneqq \frac{1}{8\pi}\left( \frac{1}{\|{x}-{c}_i\|_2}I +
\frac{
({x}-{c}_i)
\otimes 
({x}-{c}_i)
}{\|{x}-{c}_i\|_2^3}
\right){\alpha}_i, 
\end{equation}
\begin{equation}
    R_i \beta_i \coloneqq \frac{1}{8\pi\|{x}-{c}_i\|_2^3}\beta_i\times ({x}-{c}_i).
\end{equation}}
\jrev{We use the following Dirichlet data as boundary conditions
\begin{equation}\label{eq:ex4data}
f(x) = \begin{cases}
(0,0,1) & x\in \Gamma_0 \\
(0,0,0) & x\notin\Gamma_0
\end{cases}.
\end{equation}
}
\jrev{We discretize using triangulations of the boundaries and using the surface areas of the triangles as integration weights. The interior hole has radius equal to 0.1 times that of the outer boundary, and is discretized with about 0.01-0.03 times the number of triangles on the outer boundary. \secondrev{The number of triangles on the cow is fixed at 5,856.} The proxy surfaces are discretized using trapezoidal quadrature for the azimuthal parameter and Gauss-Legendre quadrature for the polar parameter. The reason for using different discretizations for the boundaries versus the proxy surface is that we wish to keep the boundary discretization easily generalizeable to different meshes, whereas it is typically reasonable to use a sphere as the proxy surface.}

\jrev{
We see an increased opportunity for parallelism in factorization in three dimensions\textemdash Figure~\ref{fig:sphere_par_plot} shows a greater scaling with number of threads than we saw in \secondrev{Section~\ref{sec:ex1}}, as expected due to the greater relative cost per level of higher dimensions described in Theorem~\ref{thm:rscost}. However, due to the fraction of work contained at the root node we see \secondrev{less significant} gains in parallelism within the linear solve. As the factorization is orders of magnitude more expensive than the solves, we believe it is far more important to explore parallelism in that process. Furthermore, if the solution to many linear systems is desired, as is often the motivation for using a direct solver, Figure~\ref{fig:sphere_par_plot} suggests that it is better to simply compute the linear solves in parallel rather than trying to parallelize each solve.}

\thirdrev{We remark that the timing results in Figure~\ref{fig:sphere_par_plot} for one thread indicate an approximately linear scaling in the number of discretization points, whereas Corollary~\ref{cor:cost} establishes costs for the factorization and solve that are worse than linear. We observe this to be caused by a slower growth in the number of skeleton indices at the root node than assumed for intrinsic dimension $d=2$ in \eqref{eq:kl}. More generally, the extent to which the skeleton growth tightly follows this bound can be highly geometry-dependent and, in this experiment, is likely affected by the fixed number of discretization points used for the cow across all problem sizes.}

\secondrev{After the initial factorization, we modify the boundary by moving the interior spherical hole downwards by a distance equal to 0.2 times the outer boundary radius, and the factorization is updated based on this perturbation. In Figure~\ref{fig:sphere_upd_plot} we see that updating the factorization based on the technique in Section~\ref{sec:factupdate} results in a notable speedup over recomputing the factorization from scratch, although the speedup is less substantial than in Figure~\ref{fig:ex1speedup}. The cost of computing a linear solve in the perturbed geometry is effectively the same regardless of whether that factorization is from scratch or from an update.}

\secondrev{\thirdrev{In this experiment, the gains seen from updating are smaller than seen in Figure~\ref{fig:ex1speedup} due to the fact that the relative amount of work done in factoring the root node is greater in this case. Furthermore, when the skeleton growth follows the assumption in~\eqref{eq:kl}, this is the expected behavior with increasing dimension.} Using the recompression techniques of HIF-IE can mitigate this, although there is a tradeoff\textemdash updating a HIF-IE factorization requires recompressing a greater number of tree nodes. For example, if only one node at the second level of a 3D tree contains perturbed points, then updating the recursive skeletonization we implemented will require recompression of up to 27 nodes at that level, whereas updating a HIF-IE factorization will require recompression of up to 64 nodes\textemdash the entire level. To address this,} we are currently working on developing techniques which would not require recomputations near the top of the tree.

\begin{figure}[ht]
  \centering
  \includegraphics[scale=0.33]{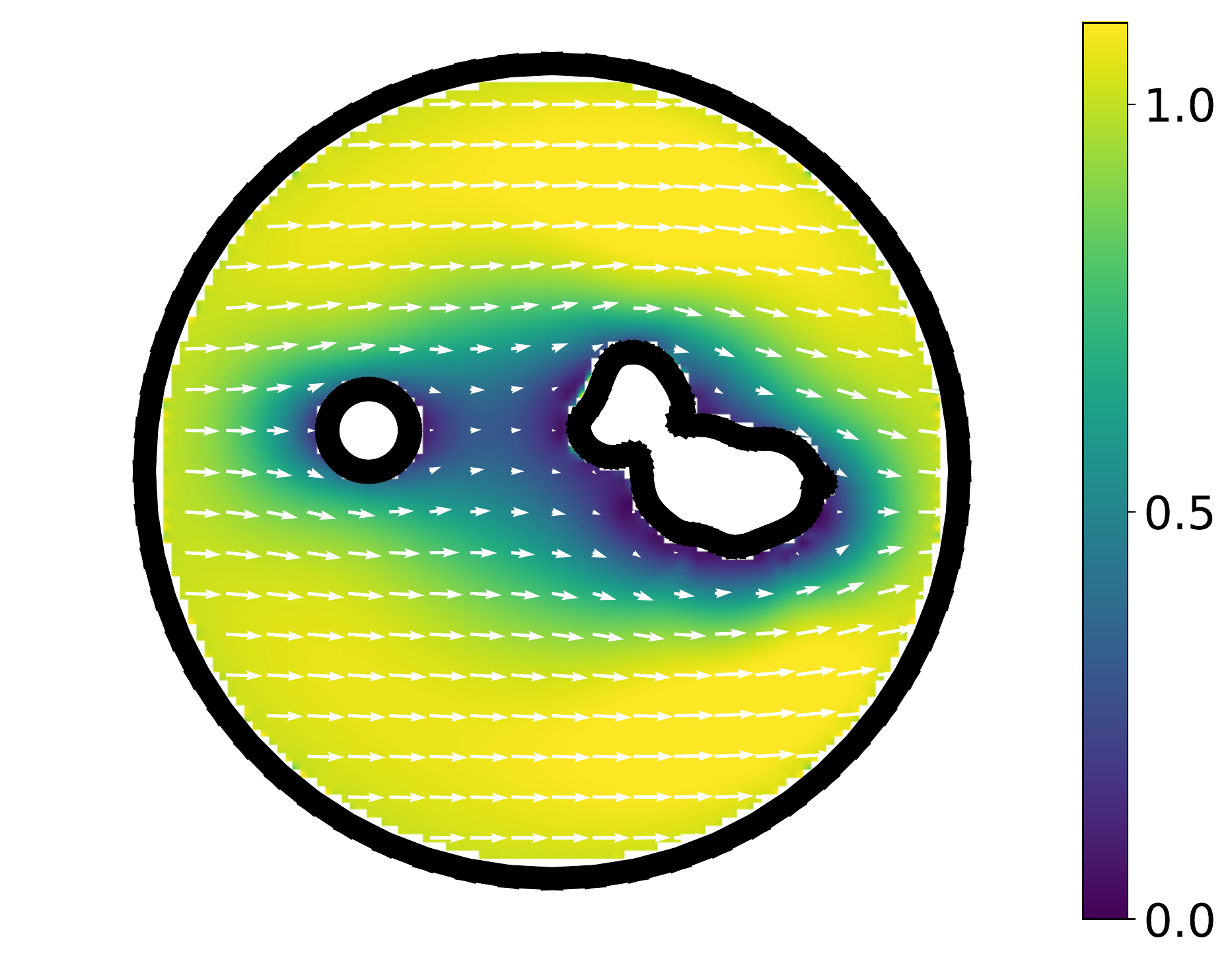}
  \includegraphics[scale=0.33]{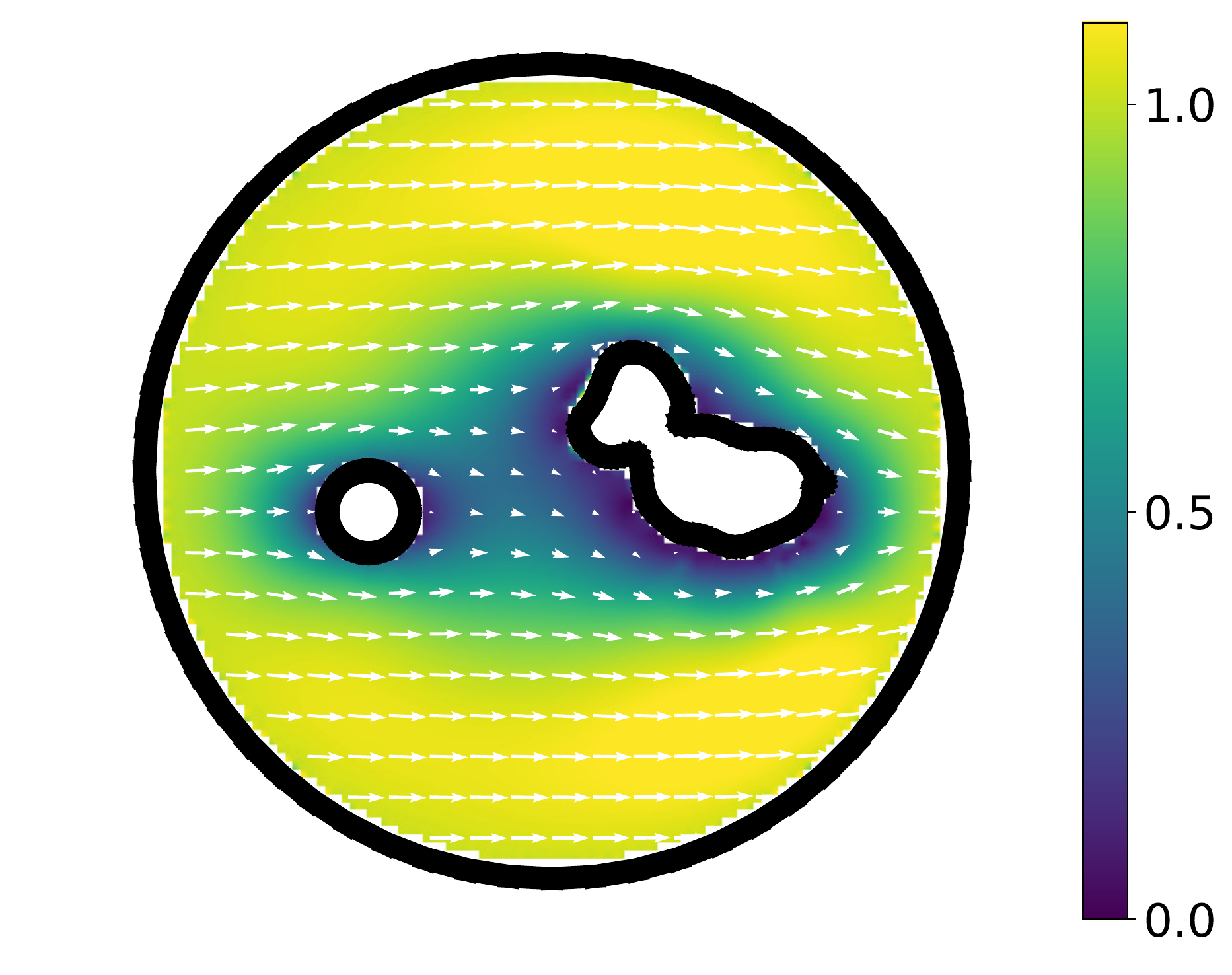}
  \caption{\secondrev{Cross-section visualizations of Stokes flow in a 3D domain. For the above, we used 14,470 discretization nodes on the outer sphere, 570 on the inner sphere, and 5,856 on the cow. On the left and right are solutions computed before and after the boundary update described in Section~\ref{sec:cowupd}.}}\label{fig:cow_cross}
\end{figure}

\begin{figure}[ht]
  \centering
  \includegraphics[trim=550 200 500 100,clip=true,scale=0.25]{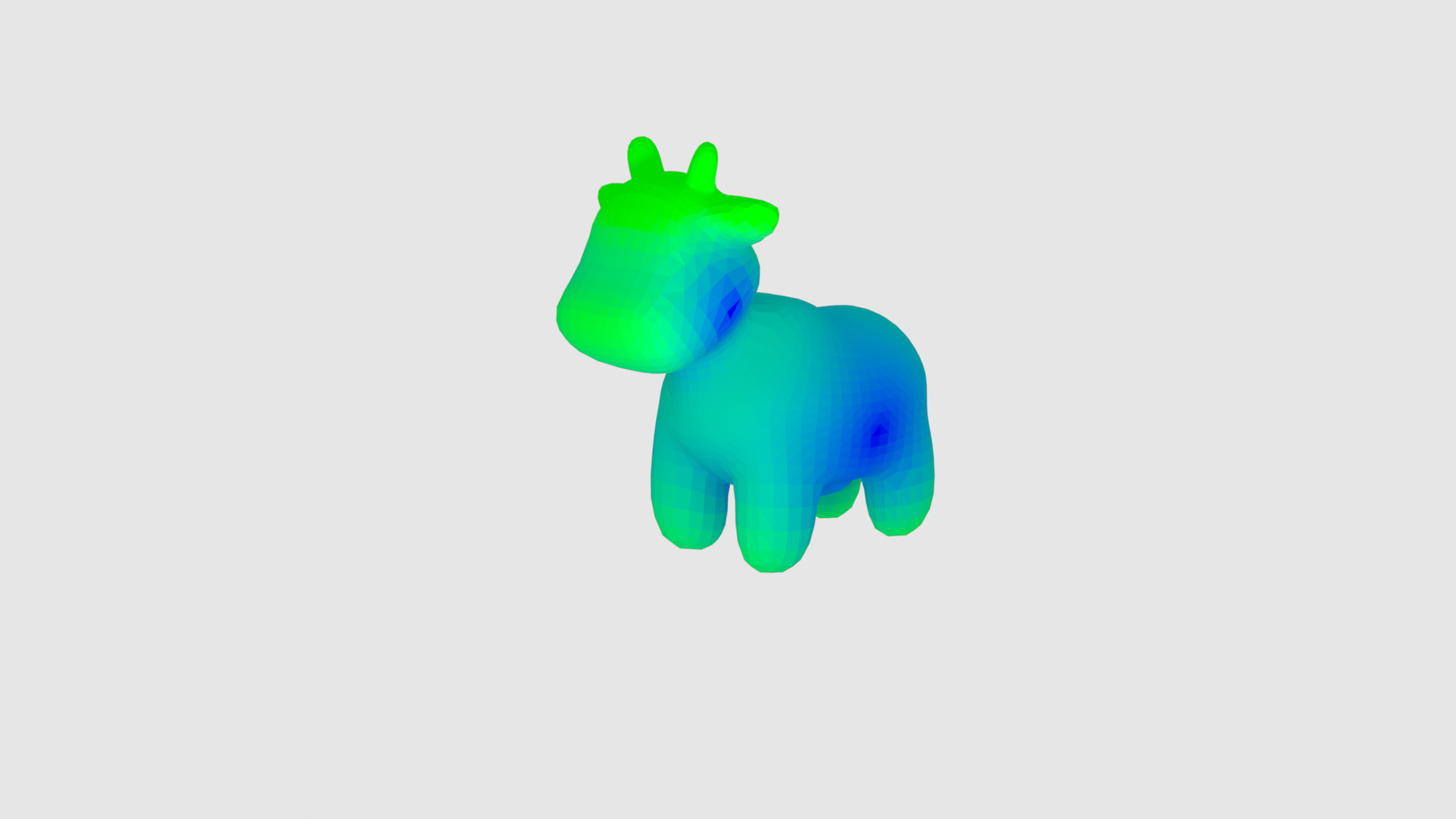}  \caption{\secondrev{The cow-shaped interior hole used in the problem described in Section~\ref{sec:cowupd}. The color corresponds to the magnitude of the double-layer potential found on the cow-shaped boundary in Figure~\ref{fig:cow_cross} \textemdash the magnitude is larger in green regions and smaller in blue regions.}}\label{fig:cow}
\end{figure}

\begin{figure}[ht]
  \centering
  \includegraphics[scale=0.25]{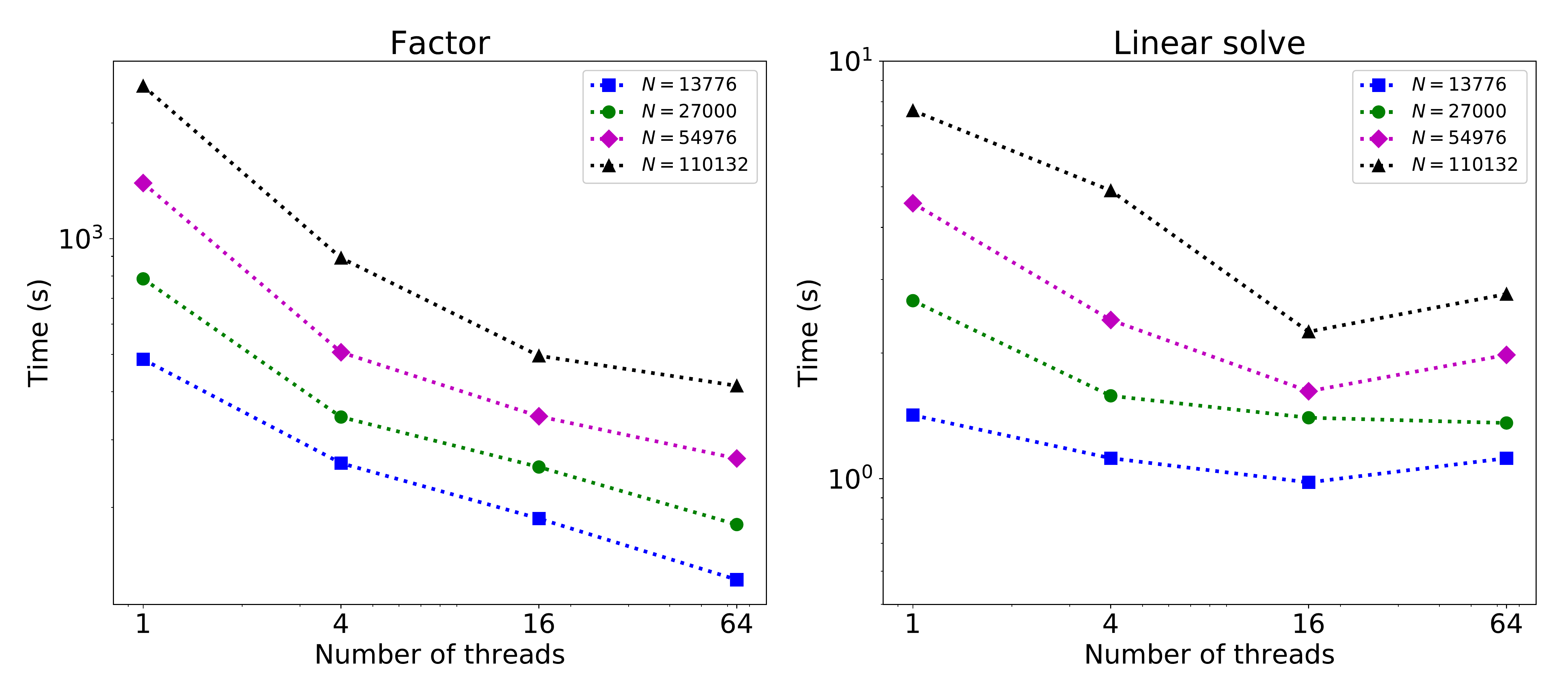}
  \caption{\jrev{Parallelizing the initial factorization for \secondrev{the system described in Section~\ref{sec:cowupd}} results in notable speedups up to \secondrev{64 threads}. In our experiments, the linear solves exhibited \secondrev{less} benefit from parallelizing, presumably due to the fraction of the work involving the root node.}}\label{fig:sphere_par_plot}
\end{figure}
\begin{figure}[ht]
  \centering
  \includegraphics[scale=0.3]{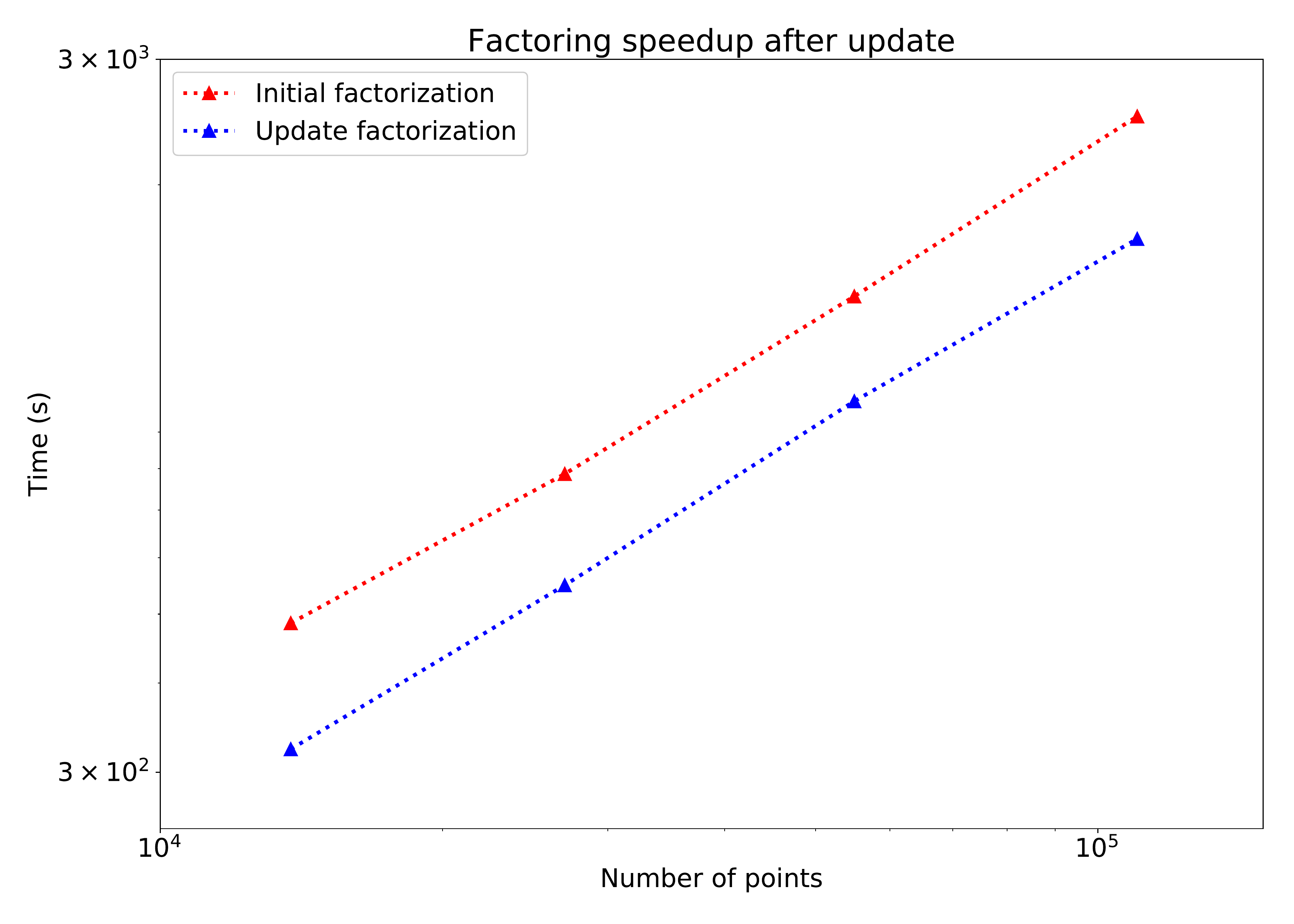}
    \caption{\secondrev{The updating scheme of Section~\ref{sec:factupdate} results in a 1.5x speedup in the factorization of the linear system described in Section~\ref{sec:ex1}. Plotted is the speedup following movement of an interior hole. When scaling the total number of points, the ratio of points on the larger sphere to points on the smaller sphere is fixed, as is the total number of points on the cow. This means that as the number of points grows we are updating a fixed fraction of the domain.}} \label{fig:sphere_upd_plot}
\end{figure}

\section{Conclusions}\label{sec:conclusions}

In this paper, we have demonstrated the applicability of skeletonization factorizations amenable to fast updating in solving large numbers of related boundary value problems. This can occur, for example, in geometry optimization or time dependent problems. Furthermore, we developed a \jrev{novel} approach to solving problems where the kernel \jrev{is singular and its} discretization is only a subblock of the whole system. This occurs, for example, when the domain is multiply-connected, and the relevant integral operators contain non-trivial null-spaces. The efficiency and parallelizability of our routines show great promise in areas where kernel matrices need to be factored multiple times following small updates to the underlying data points. Relevant areas include Gaussian process problems \cite{mle, greengaussian}, unsteady fluid simulations \cite{navier}, and shape optimization of elastic structures \cite{elastic}.

While we take one specific approach to selecting skeleton points during compression, other methods may allow for greater parallelism between levels and reduce the number of computations needed to handle geometric perturbations.

Since the updating strategy currently relies on the locality of perturbations for its efficiency, further work in this area should include improving techniques for global geometry updates so that, e.g., gradient-based optimization routines do not suffer from updating the entire boundary between steps. This may require a fundamental change to the way leaf-level compression is performed and is the subject of ongoing work.

\section*{Acknowledgments}\jrev{
The authors thank Victor Minden for several helpful conversations, \secondrev{Keenan Crane for making the cow mesh freely available~\cite{Model3D},} and the anonymous reviewers whose many thoughtful comments helped improve this manuscript.}
\bibliographystyle{siamplain}
\bibliography{references}
\end{document}